\RequirePackage{fix-cm}
\documentclass[12pt]{amsart}
\usepackage[utf8]{inputenc}
\usepackage[T1]{fontenc}

\usepackage{graphicx}
\usepackage{longtable}
\usepackage{float}
\usepackage{wrapfig}
\usepackage{textcomp}
\usepackage{marvosym}
\usepackage{latexsym}
\usepackage{amssymb}
\usepackage{xargs}
\usepackage{fullpage}
\usepackage[sort&compress,numbers]{natbib}
\usepackage{algorithm2e}
\usepackage{xcolor}
\usepackage{subcaption}
\definecolor{HalfGray}{gray}{0.55}
\definecolor{OliveGreen}{rgb}{0,.35,0}
\definecolor{webbrown}{rgb}{.6,0,0}
\definecolor{BrightViolet}{rgb}{0.5,0.2,0.8}
\definecolor{Maroon}{cmyk}{0, 0.87, 0.68, 0.32}
\definecolor{RoyalBlue}{cmyk}{1, 0.50, 0, 0.25}
\definecolor{Black}{cmyk}{0, 0, 0, 0}

\definecolor{ccccccc}{RGB}{204,204,204}
\definecolor{c808080}{RGB}{128,128,128}
\definecolor{c999999}{RGB}{153,153,153}
\definecolor{ce6e6e6}{RGB}{230,230,230}

\usepackage[unicode]{hyperref}
\hypersetup{
colorlinks=true,
linktocpage=true,
pdfstartview=FitH,
breaklinks=true,
pdfpagemode=UseNone,
pageanchor=true,
pdfpagemode=UseOutlines,
plainpages=false,
bookmarksnumbered,
bookmarksopen=false,
bookmarksopenlevel=1,
hypertexnames=true,
pdfhighlight=/O,
urlcolor=Maroon, linkcolor=RoyalBlue, citecolor=OliveGreen, 
pdftitle={},
pdfauthor={},
pdfsubject={},
pdfkeywords={},
pdfcreator={pdfLaTeX},}
\usepackage{pgfplots}
\usepackage{amsthm,amsmath,enumerate,bbm,tikz,tabularx,multicol,breqn,environ}
\usetikzlibrary{matrix,arrows,calc,patterns}
\usepgfplotslibrary{groupplots,dateplot}
\usetikzlibrary{patterns,shapes.arrows}
\pgfplotsset{compat=newest}

\makeatletter
\newsavebox{\measure@tikzpicture}
\NewEnviron{scaletikzpicturetowidth}[1]{%
  \def\tikz@width{#1}%
  \def\tikzscale{1}\begin{lrbox}{\measure@tikzpicture}%
  \BODY
  \end{lrbox}%
  \pgfmathparse{#1/\wd\measure@tikzpicture}%
  \edef\tikzscale{\pgfmathresult}%
  \BODY
}
\makeatother

\newtheorem{theorem}{Theorem}[section]
\newtheorem{proposition}[theorem]{Proposition}
\newtheorem{corollary}[theorem]{Corollary}
\newtheorem{lemma}[theorem]{Lemma}
\theoremstyle{definition}
\newtheorem{definition}[theorem]{Definition}

\newtheorem{example}[theorem]{Example}
\theoremstyle{remark}
\newtheorem{remark}[theorem]{Remark}

\DeclareMathOperator*{\argmin}{arg\,min}

\renewcommand{\bar}{\overline}
\renewcommand{\geq}{\geqslant}
\renewcommand{\leq}{\leqslant}
\newcommand{\ps}[2]{\left\langle #1 \vert #2 \right\rangle}
\def\eqsp{\,}

\def\UMD{\textrm{UMD}}
\newcommand{\ensemble}[2]{\left\{#1\,:\eqsp #2\right\}}
\newcommand{\set}[2]{\ensemble{#1}{#2}}
\newcommandx{\Id}[1][1=n]{\operatorname{I}_{#1}}
\def\eqdef{:=}

\usepackage[textsize=footnotesize]{todonotes} 

\def\vt{{\vartheta}}
\def\prox{\hbox{$\Pi$}}

\def\X{{\mathcal{X}}}
\def\N{{\mathcal{N}}}

\newcommand{\be}{\begin{eqnarray}}
\newcommand{\ee}[1]{\label{#1}\end{eqnarray}}
\newcommand{\nn}{\nonumber \\}
\newcommand{\ese}{\end{eqnarray*}}
\newcommand{\bse}{\begin{eqnarray*}}

\newcommand{\half}{ \mbox{$\frac{1}{2}$}}
\def\rset{\mathbb{R}}

\renewcommand{\leq}{\leqslant}
\renewcommand{\geq}{\geqslant}
\DeclareMathOperator*{\Argmin}{Arg\,min}
\DeclareMathOperator*{\Argmax}{Arg\,max}
\def\qed{\hfill$\Box$}

\def\rset{\mathbb{R}}

\author{Anatoli Juditsky \and Joon Kwon \and Éric Moulines}
\date{\today}
\title{Unifying mirror descent and dual averaging}
\begin{document}

\begin{abstract}
We introduce and analyze a new family of first-order optimization algorithms which generalizes and unifies both mirror descent and dual averaging. Within the framework of this family, we define new algorithms for constrained optimization that combines the advantages of mirror descent and dual averaging. Our preliminary simulation study shows that these new algorithms significantly outperform available methods in some situations.
\end{abstract}
\maketitle

\setcounter{tocdepth}{1}
\tableofcontents

\section{Introduction}
Mirror descent algorithms were initially introduced as
first-order convex optimization algorithms, and were then extended to a
variety of (online) optimization problems.
Let us quickly recall the succession of ideas which have led to the
mirror descent algorithms.

Let us start with the most basic setting, in which the objective function
{$f:\,\rset^n\to \rset$} is convex on
$\rset^n$, differentiable, and admits a unique minimizer $x_*\in
\rset^n$. We focus on the construction of algorithms based on
first-order oracles (in other words, the algorithm is allowed to query the values of the objective $f(x_t)$
and of its gradient $\nabla f(x_t)$ at a search points $x_t\in \rset^n$)
and which outputs points where the value of the objective function
$f$ is provably close to the minimum $f_*=f(x_*)$. The most basic of such
algorithm is the (Euclidean) gradient descent, which starts at some point
$x_1\in \rset^n$ and iterates
\[ x_{t+1}=x_t-\gamma\nabla f(x_t),\quad t\geqslant 1, \]
where $\gamma > 0$ is the step-size. An equivalent way of
writing the above is the so-called \emph{proximal} formulation:
\[x_{t+1}= \argmin_{x\in \rset^n}\left\{ f(x_t)+\ps{\nabla f(x_t)}{x-x_t} +\frac{1}{2\gamma}\left\| x-x_t \right\|_2^2 \right\}, \]
where $x_{t+1}$ appears as the solution of a simplified minimization
problem where the objective function $f$ has been replaced by its linearization at $x_t$
\emph{plus} a Euclidean \emph{proximal term} $\frac{1}{2\gamma}\left\| x-x_t
\right\|_2^2$ which prevents the next iterate $x_{t+1}$ from being too far from
$x_t$. This algorithm is well-suited to assumptions regarding the
objective function $f$ which involve the Euclidean norm (e.g.\ if $\nabla f$
is bounded (or Lipschitz-continuous) with respect to the Euclidean norm).

The mirror descent algorithm, introduced in
\cite{nemirovski1979efficient,nemirovsky1983problem}, can be seen as an extension
of the above gradient descent, in which Euclidean proximal term is
replaced with a
\emph{Bregman divergence} \cite{bregman1967relaxation}:
\[ x_{t+1}=\argmin_{x\in \rset^n}\left\{ f(x_t)+\ps{\nabla f(x_t)}{x-x_t} +\frac{1}{\gamma}D_F(x,x_t) \right\}, \]
where, for any $x,x' \in \rset^n$ the Bregman divergence $D_F(\cdot,\cdot)$ associated
with  a differentiable strictly convex function $F:\rset^n\to \rset$ is defined as
\[ D_F(x',x):=F(x')-F(x)-\ps{\nabla F(x)}{x'-x} \eqsp. \]
Such algorithms, with a carefully chosen function $F$,
are used to better suit the geometry of the problem, for instance when
the objective function $f$ is Lipschitz-continuous or smooth with
respect to a non-Euclidean norm. One can see that the above mirror
descent iteration can be equivalently written, under appropriate assumptions on $F$,
\begin{equation}
\label{eq:primal-dual}
x_{t+1}=\nabla F^*(\nabla F(x_t)-\gamma\nabla f(x_t)),
\end{equation}
where $F^*$ is the Fenchel--Legendre transform of $F$,
\[ F^*(\vt)=\max_{x\in \rset^n}\left\{ \ps{\vt}{x}- F(x) \right\},\qquad \vt \in \rset^n.  \]

This formulation makes explicit the distinction between
the primal space of iterates $(x_t)_{t\geqslant 1}$ and
the dual space where the gradients $(\nabla f(x_t))_{t\geqslant 1}$
belong: search point $x_t$ is mapped from the primal into the dual space using
$\nabla F$, the gradient step is then performed in the dual space
($\nabla F(x_t)-\gamma\nabla f(x_t)$), and the point thus obtained is
finally mapped back into the primal space using $\nabla F^*$.

We now move on to \emph{constrained} problems. Let $\X\subset \rset^n$
be a closed convex set.  To force the
trajectory of the method to stay in $\X$, the mirror descent should be properly
adapted. Such modification can be implemented in at least two ways, which give
rise to two families of algorithms: \emph{mirror descent} (MD) which
can be traced back to the pioneering work
\cite[Chapter~3]{nemirovsky1983problem} and \emph{dual averaging} (DA)
introduced in \cite{nesterov2009primal,juditsky2005recursive}, sometimes called {\em lazy} mirror descent.
To illustrate the similarities and differences between MD\ and DA, we
here describe their implementation in the simple Euclidean case. The MD\ algorithm in this case corresponds to
the \emph{projected gradient descent}, 
in which, given an initial point $x_1\in \X$,  for $t\geqslant 1$,
 \[ y_{t+1}=x_t-\gamma\nabla f(x_t)\quad \text{and}\quad x_{t+1}=\operatorname{proj}_{\X}(y_{t+1}), \]
 where $\operatorname{proj}_{\X}$ denotes the Euclidean
 projection onto $\X$. In other words, it first performs a
 gradient step, then projects the point thus obtained onto the set
 $\X$; then the next gradient step is performed starting from $x_{t+1}$, and so on.

For a given initial point $\vt_1\in \rset^n$, the corresponding algorithm in
the DA\ family writes, for all $t\geqslant 1$:
\[ \vt_{t+1}=\vt_t-\gamma\nabla  f(x_t)\quad \text{and}\quad x_{t+1}=\operatorname{proj}_{\X}(\vt_{t+1}). \]
The difference with the projected gradient descent is that the
gradient increment is performed  from the \emph{unprojected} point $\vt_t$.

The MD\ and DA\ algorithms share similarities in their analysis
and in the guarantees they provide. However, their differences led
to the two families of algorithms being used and studied in different situations. For instance,
DA\  algorithms seem to be advantageous in distributed
problems \cite{dekel2012optimal,duchi2012dual}, and manifold
identification \cite{lee2012manifold,duchi2016local}. They are also believed to possess better averaging properties in the presence
of noise \cite{flammarion2017stochastic}.
On the other hand, MD\ is shown to provide better convergence rates in
some cases (e.g., in the case of smooth objective $f$, cf. \cite[Section~4.2]{flammarion2017stochastic}). MD\ also achieves the optimal rate in the adversarial multi-armed bandit problem \cite{audibert2009minimax,audibert2010regret} and the online combinatorial optimization problem with semi-bandit \cite{audibert2013regret} and bandit feedback \cite{bubeck2012towards,cohen2017tight}.

Algorithms of the mirror descent type were also transposed to provide solutions
for other problems with convex structure. We already mentioned
bandit problems. More generally, the mirror descent algorithms have
been successful in online learning, see
e.g.~\cite{zinkevich2003online,cesa2006prediction,shalev2007online,shalev2011online,rakhlin2009lecture,bubeck2011introduction,hazan2012convex}, when solving saddle-point problems and variational inequalities
\cite{nemirovski2004prox,nesterov2007dual,juditsky2011_2}, similar procedures were used
for estimator aggregation in statistical learning
\cite{juditsky2005recursive,juditsky2008learning}, etc.\footnote{We also refer to \cite[Appendix C]{mcmahan2017survey} for a discussion
comparing MD\ and DA.}

\paragraph{Main contribution.} In this paper, we introduce and study a new family of algorithms, which we refer to as
\emph{unified mirror descent} (UMD) which unifies and extends both
mirror descent and dual averaging. The general algorithm has the
property of offering at each step a set of possible iterations. 
We also construct, in the context of this new family, two algorithms
for constrained optimization we refer to as \emph{alternating
  primal-dual descent} (APDD) and \emph{interpolating primal-dual
  descent} (IPDD) capable of outperforming mirror descent and dual
averaging algorithms in some situations.

It should be mentioned that MD\ and DA\ algorithms were studied in a common framework in \cite{mcmahan2011follow,mcmahan2017survey}. Those works are, however, are very different in spirit---they deal with  unconstrained problems, the
difference between mirror descent and dual averaging appears as a
result of utilizing the regularizers/mirror maps which vary over time, and the
unification is then achieved by tweaking the way the time-varying
regularizers/mirror maps are defined.

\paragraph{Paper outline.}
In Section~\ref{sec:greedy-mirr-desc}
(resp.~\ref{sec:lazy-mirror-descent}) we recall the definitions of the
mirror descent (resp.\ dual averaging) algorithms.  Then, we introduce in
Section~\ref{sec:unif-mirr-desc} a new family of
algorithms called UMD and
show that MD and DA are special cases. We then establish some basic
properties of UMD, to be used to
derive complexity estimates for UMD-type algorithms
in various contexts. Next, in
Section~\ref{sec:guar-conv-nonc}, we study applications of UMD to smooth and nonsmooth convex
optimization. 
In Section~\ref{sec:gold:-new-algorithm} we introduce two
new algorithms---APDD (alternating primal-dual descent) and IPDD (interpolating primal-dual descent)---and present results of some preliminary numerical experiments which
show that they compare favorably to MD and DA.

Proofs which are longer than few lines are postponed to Appendix \ref{sec:postponed-proofs}.
\subsection{Preliminaries and notation}
\label{sec:preliminairies}
For $x,\vt \in \rset^n$, $\ps{\vt }{x}$ stands for the canonical scalar product.
For a given set $\mathcal{A}\subset \rset^n$,
$\operatorname{int}\mathcal{A}$ and $\operatorname{cl}\mathcal{A}$
denote its interior and closure respectively.
For a given norm $\left\|\,\cdot\,\right\|_{}$ on $\rset^n$, we
denote $\left\|\,\cdot\,\right\|_{*}$ the conjugate norm,
\[
\left\| \vt  \right\|_{*}=\max_{\left\| x \right\| \leqslant 1} \ps{\vt}{x} ,\qquad \vartheta\in \mathbb{R}^n.
\]
The characteristic function $I_{\mathcal{C}}:\rset^n\to \rset\cup\{+\infty\}$ of
a convex set $\mathcal{C}\subset \rset^n$ is zero on $\mathcal{C}$ and
equal to $+\infty$ elsewhere.
Let $g:\rset^n\to \rset\cup\{+\infty\}$ be a convex function. Its domain
$\operatorname{dom}g$ is the set $\set{x \in \rset^n}{g(x) < \infty}$. Its
subdifferential $\partial g(x)$ at  $x\in \rset^n$ is the set of $\vt \in \rset^n$ such that
  \[ \forall x'\in \rset^n,\quad g(x')-g(x)\geqslant \ps{\vt}{x'-x}\eqsp, \]
  we refer to $\vt\in \partial g(x)$ as subgradients of $g$ at $x$.
  The Fenchel--Legendre transform of $g$ is defined by
  \[ g^*(\vt)=\max_{x\in \rset^n}\left\{ \ps{\vt}{x} -g(x) \right\},\qquad \vt \in \rset^n.  \]
If $g$ is differentiable at $x\in \rset^n$, its Bregman
  divergence between $x,x'\in \rset^n$ is defined as
\[
D_g(x',x)=g(x')-g(x)-\ps{\nabla g(x)}{x'-x} \eqsp.
\]
Some convexity definitions and results are recalled in Section~\ref{sec:conv-analys-tools}.

Throughout the paper, we consider
algorithms associated with an arbitrary sequence
$(\xi_t)_{t\geqslant 1}$ in $\rset^n$, and the problem domain $\X\subset \rset^n$ is a closed and nonempty convex set.

\section{Mirror descent and dual averaging algorithms}
\label{sec:MD-DA}
\subsection{Mirror descent}
\label{sec:greedy-mirr-desc}
The mirror descent algorithms rely on the notion of
\emph{mirror maps} that we now recall. Our presentation draws
inspiration from \cite{bubeck2015convex}, with a few differences in definitions and
conventions.
\begin{definition}
  \label{def:mirror-maps}
  Let $F:\rset^n\to \rset\cup\{+\infty\}$. Denote $\mathcal{D}_F:=\operatorname{int}\operatorname{dom}F$. We say that
$F$ is an $\X$-compatible \emph{mirror map} if
\begin{enumerate}[(i)]
\item\label{item:F-lsc-strict-convex} $F$ is lower-semicontinuous and strictly convex,
\item\label{item:F-diff} $F$ is differentiable on $\mathcal{D}_F$,
\item\label{item:gradient-all-values} the gradient of $F$ takes all possible values, i.e.\ $\nabla F(\mathcal{D}_F)=\rset^n$.
\item\label{item:X-subset-bar-D} $\X\subset \operatorname{cl}\mathcal{D}_F$,
\item\label{item:X-inter-D-nonempty} $\X\cap \mathcal{D}_F\neq \emptyset$.
\end{enumerate}
\end{definition}
The following proposition gathers a few properties of mirror maps.
For the sake of completeness, its proof is given in Appendix~\ref{sec:postponed-proofs}.

\begin{proposition}
  \label{prop:mirror-map-properties}
Let $F:\rset^n\to \rset\cup\{+\infty\}$ be an
$\X$-compatible mirror map, $F^*$ the Fenchel--Legendre transform of $F$, and $\mathcal{D}_F:=\operatorname{int}\operatorname{dom}F$. Then,
\begin{enumerate}[(i)]
\item\label{item:F-star-dom} $\operatorname{dom}F^*=\rset^n$,
\item\label{item:F-star-diff} $F^*$ is differentiable on $\rset^n$,
\item\label{item:F-star-image} $\nabla F^*(\rset^n)=\mathcal{D}_F$,
\item\label{item:gradient-F-star-inverse} For all $x\in \mathcal{D}_F$ and $y\in \rset^n$, $\nabla F^*(\nabla F(x))=x$ and $\nabla F(\nabla F^*(y))=y$.
\end{enumerate}
\end{proposition}

We can now define the mirror descent algorithm
\cite{nemirovski1979efficient}.
\begin{definition}
Let $F$ be an $\X$-compatible mirror map, $x_1\in
\X\cap \mathcal{D}_F$, and
$\xi\eqdef(\xi_t)_{t\geqslant 1}$ be a sequence in $\rset^n$. We define the
associated MD\  iterates according to
\begin{equation}
\label{eq:gmd}
\tag{MD}
x_{t+1}=\argmin_{x\in \X}D_{F}(x,\ \nabla F^*(\nabla F(x_t)+\xi_t)),\qquad t\geqslant 1.
\end{equation}
$(x_t)_{t\geqslant 1}$ is then said to be an MD$(\X,F,\xi)$ sequence
and $\xi:=(\xi_t)_{t\geqslant 1}$ is called the sequence of \emph{dual increments}.
\end{definition}
The above is well-defined thanks to the following recursive argument.  As soon
as $x_t$ ($t\geqslant 1$) belongs to $\X\cap \mathcal{D}_F$,
$\nabla F(x_t)$ exists because $F$ is differentiable on
$\mathcal{D}_F$ by Definition~\ref{def:mirror-maps}. Then,
$\nabla F^*(\nabla F(x_t)+\xi_t)$ exists because $F^*$ is
differentiable on $\rset^n$ by
Proposition~\ref{prop:mirror-map-properties}--(\ref{item:F-star-diff}).
Then, the next iterate $x_{t+1}$ is obtained using the Bregman
projection onto $\X$, which is well-defined and belongs to
$\X\cap \mathcal{D}_F$ thanks to Theorem~\ref{thm:bregman-projection}
below.
\begin{theorem}[Bregman projection onto $\X$]\footnote{Similar statements can be found in the
literature (cf.\ e.g., \cite[Lemma A.1]{cox2014dual}) but we could not
find one that exactly matches assumptions of Theorem \ref{thm:bregman-projection} on $F$ and $\X$. We provide a detailed proof in
Appendix~\ref{sec:postponed-proofs} for completeness.
}
  \label{thm:bregman-projection}
  Let $F:\rset^n\to \rset\cup\{+\infty\}$ be an
  $\X$-compatible mirror map, and let
  $\mathcal{D}_F:=\operatorname{int}\operatorname{dom}F$. Then for any point
  $x_0\in \mathcal{D}_F$, the minimizer of $D_F(x,x_0)$ over $x\in \X$ exists, is unique, and  belongs to $\X\cap \mathcal{D}_F$.
  In other words,
  \[ \argmin_{x\in \X}D_F(x,x_0)=\argmin_{x\in \X\cap \mathcal{D}_F}D_F(x,x_0). \]
\end{theorem}

The above \eqref{eq:gmd} iteration can be rewritten as
follows. Denote $\tilde{x}_t:=\nabla F^*(\nabla F(x_t)+\xi_t)$.  Using
Proposition~\ref{prop:mirror-map-properties}--(\ref{item:gradient-F-star-inverse}),
we get $\nabla F(\tilde{x}_t)=\nabla F(x_t)+\xi_t$. Then, using the
definition of the Bregman divergence:
\begin{align*}
x_{t+1}&=\argmin_{x\in \X}\left\{ F(x)-F(\tilde{x}_t)-\left<
  \nabla F(\tilde{x}_t) \middle| x-\tilde{x}_t \right>  \right\}\\ &=\argmin_{x\in \X}\left\{ F(x)-\left< \nabla F(x_t)+\xi_t
                                                                     \middle| x \right>  \right\}\\
&=\argmin_{x\in \X}\left\{ F(x)-F(x_t)-\left< \nabla F(x_t)
                                                                                                      \middle| x-x_t \right> -\left< \xi_t \middle| x \right>  \right\}\\
&=\argmin_{x\in \X}\left\{ -\left< \xi_t \middle| x \right> +D_F(x,x_t) \right\}.
\end{align*}
The last expression is called \emph{primal formulation}, and is usually
taken as the definition of mirror descent, cf. \cite[Section~3]{beck2003mirror}. Introducing the
\emph{prox-mapping}:
\[ T_{\X,F}(u,x):=\argmin_{x'\in \X}\left\{ -\left< u \middle| x' \right> +D_F(x',x) \right\},\quad u\in \rset^n,\ x\in \X\cap \mathcal{D}_F, \]
the MD\ iterates starting from some $x_1\in \X\cap \mathcal{D}_F$ can then be alternatively written as:
\begin{equation}
\label{eq:md-prox}
\tag{MD-prox}
x_{t+1}=T_{\X,F}(\xi_t,x_t),\quad t\geqslant 1.
\end{equation}

Some examples of widely used mirror maps are as follows.
\begin{example}[Gradient descent]
  The simplest example is provided by
$\X=\mathbb{R}^n$ and $F(x)=\frac{1}{2}\left\| x
\right\|_2^2$. One can easily see that $F$ is indeed an
$\mathbb{R}^n$-compatible mirror map. In this case, $\nabla F(x)=\nabla F^*(x)=x$ are identity mappings, and the update rule \eqref{eq:gmd} boils down to the gradient
descent iteration if we consider dual increments $\xi_t:=-\gamma\nabla
f(x_t)$ where $f:\mathbb{R}^n\to \mathbb{R}$ is a
differentiable objective function.
\end{example}
\begin{example}[Projected gradient descent]
A common variant of the previous example is the case where $\X$ is some
closed proper subset of $\mathbb{R}^n$. One can easily see that
$F(x)=\frac{1}{2}\left\| x \right\|_2^2$ is an $\X$-compatible
mirror map. If $f$ is an objective function which is
differentiable on $\X$, then \eqref{eq:gmd} with $\xi_t=-\gamma\nabla f(x_t)$
corresponds to the standard projected gradient descent algorithm.
\end{example}
\begin{example}[Exponential weights]
  A special case corresponds to $\X$ being the $n$-simplex:
  \[ \X=\left\{ x\in \mathbb{R}^n_+\, \colon   \sum_{i=1}^nx_i=1 \right\}  \]
  and $F$ given by $F(x)=\sum_{i=1}^nx_i\log x_i$ for $x\in \mathbb{R}_+^n$
  (using convention $0\log 0=0$) and $F(x)=+\infty$ for $x\not \in
  \mathbb{R}_+^n$. Then $F$ is  an $\X$-compatible mirror map.
\end{example}

\begin{figure}
  \centering
\begin{tikzpicture}[y=0.80pt, x=0.80pt, yscale=-1.000000, xscale=1.000000, inner sep=0pt, outer sep=0pt,scale=2]
  \path[draw=black,fill=ccccccc,line join=miter,line cap=butt,even odd rule,line
    width=0.800pt] (-24.9259,-39.9663) .. controls (-10.0099,-39.7633) and
    (5.8309,-23.3238) .. (0.0000,0.0000) .. controls (-4.0000,16.0000) and
    (-15.5225,33.4460) .. (-29.1922,30.4615);
\draw (0,0) node {$\bullet$};
\draw[->,>=latex] (-15,-15) node {$\bullet$} node[above=5pt]{$x_t$}
to[bend right=20] node[midway,above=5pt]{$\nabla F$} (115,-15) node
{$\bullet$} node[right=5pt] {$\nabla F(x_t)$};
\draw[->,>=latex] (115,-15) to (100,10) node
{$\bullet$} node[below right=5pt] {$\nabla F(x_t)+\xi_t$};
\draw[->,>=latex] (100,10) to[bend right=20] node[midway,below=5pt]
{$\nabla F^*$} (20,20) node {$\bullet$};
\draw[->,>=latex] (20,20) to[bend left=20] node[midway,above right]{$\displaystyle \argmin_{x\in \X}D_F(x,\,\cdot\,)$} (0,0) node{$\bullet$} node[left=5pt] {$x_{t+1}$};
\draw[dashed] (75,-50) -- (75,30);
\draw (40,-45) node {primal space};
\draw (110,-45) node {dual space};
\draw (-20,20) node {$\X$};
\end{tikzpicture}
  \caption{Mirror descent}
  \label{fig:greedy-mirror-descent}
\end{figure}

\subsection{Dual averaging}
\label{sec:lazy-mirror-descent}
The dual averaging algorithms rely on the notion of regularizers
which we now recall. These are  less restrictive than mirror maps:
we see below in Proposition~\ref{prop:F+I_X} that for a given mirror map, there always
exists a corresponding regularizer but the converse is not
true.
\begin{definition}[Regularizers]
A function $h:\rset^n\to \rset\cup\{+\infty\}$ is an $\X$-\emph{pre-regularizer} if it is strictly convex, lower-semicontinuous, and if $\operatorname{cl}\operatorname{dom}h=\X$.
Moreover, if $\operatorname{dom}h^*=\rset^n$,
then $h$ is said to be an $\X$-\emph{regularizer}.
\end{definition}
The following proposition gives several sufficient conditions for
the condition $\operatorname{dom}h^*=\rset^n$ to be
satisfied.
\begin{proposition}
  \label{prop:sufficient-conditions-dom-h-star-R}
Let $h$ be an $\X$-pre-regularizer.
\begin{enumerate}[(i)]
\item If $\X$ is compact, then $h$ is an $\X$-regularizer.
\item If $h$ is differentiable on
  $\mathcal{D}_h:=\operatorname{int}\operatorname{dom}h$ and $\nabla
  h(\mathcal{D}_h)=\rset^n$, then $h$ is an $\X$-regularizer.
\item If $h$ is strongly convex with respect to some norm
  $\left\|\,\cdot\,\right\|$, then $h$ is an $\X$-regularizer.
\end{enumerate}
\end{proposition}

\begin{proposition}[Differentiability of $h^*$]
  \label{prop:differentiability-h-star}
  Let $h$ be an $\X$-regularizer. Then $h^*$ is differentiable
  on $\rset^n$.
\end{proposition}

\begin{proposition}
  \label{prop:F+I_X}
  Let $F$ be an $\X$-compatible mirror map.
  Then, $h:=F+I_{\X}$ is an $\X$-regularizer, and, moreover,
  $\nabla F(x)\in \partial h(x)$ for all $x\in \mathcal{D}_F$.
\end{proposition}
Proofs of Propositions \ref{prop:sufficient-conditions-dom-h-star-R}--\ref{prop:F+I_X} are postponed to Appendix~\ref{sec:proofs-2}.
\begin{corollary}
\begin{enumerate}[(i)]
\item $h(x):=\frac{1}{2}\left\| x \right\|_2^2+I_{\X}(x)$ is an $\X$-regularizer.
\item The entropy defined as:
  \[ h(x):=
\begin{cases}
  \sum_{i=1}^nx_i\log x_i&\text{if $x\in \Delta_n$}\\
  +\infty&\text{otherwise},
\end{cases}
  \]
  where $\Delta_n:=\left\{ x\in \rset_+^n \,\colon  \sum_{i=1}^nx_i=1
  \right\}$ (where $0\log 0=0$ by convention) is a $\Delta_n$-regularizer.
\end{enumerate}
\end{corollary}

\begin{example}[Elastic-net regularization]
  An example of a regularizer which does not have a mirror map
  counterpart, because it fails to be differentiable,
  is the so-called \emph{elastic-net} regularizer:
  \[ h(x):=\left\| x \right\|_1+\left\| x \right\|_2^2, \] which is
  indeed a $\rset^n$-regularizer due to the strong convexity
  (cf. Proposition~\ref{prop:sufficient-conditions-dom-h-star-R}).
\end{example}

We now recall the definition of the dual averaging (DA) iterates in
the case of the constant regularizer.\footnote{In its general form \cite{nesterov2009primal}, the DA
algorithm allows for a time-variable regularizer. For the sake of clarity, we consider here the simple case of time-invariant regularizers which already captures some essential
differences between MD and DA.}
\begin{definition}[Dual averaging \cite{nesterov2009primal,nesterov2007dual}]
Let $h$ be an $\X$-regularizer and let $\xi:=(\xi_t)_{t\geqslant 1}$
be a sequence in $\rset^n$.
A sequence $(x_t,\vt_t)_{t\geqslant 1}$ is said to be a sequence of
DA\ iterates associated with $h$ and $\xi$ (DA$(h,\xi)$ for short) if for $t\geqslant 1$,
\begin{align}
\label{eq:lmd}
\tag{DA}
  \begin{split}
    x_t&=\nabla h^*(\vt_t)\\
   \vt_{t+1}&=\vt_t+\xi_t.
  \end{split}
\end{align}
Points $(x_t)_{t\geqslant 1}$ (resp.\ $(\vt_t)_{t\geqslant 1}$) are
then called \emph{primal iterates} (resp.\ \emph{dual iterates}), and
vectors $(\xi_t)_{t\geqslant 1}$ are called \emph{dual increments}.
\end{definition}

We can see that for a given couple $(x_1,\vt_1)$ of initial points
satisfying $x_1=\nabla h^*(\vt_1)$, and a sequence $(\xi_t)_{t\geqslant
  1}$ of dual increments, the subsequent iterates $(x_t,\vt_t)_{t\geqslant 2}$ are well-defined and unique.

\begin{figure}
  \centering
\begin{tikzpicture}[y=0.80pt, x=0.80pt, yscale=-1.000000, xscale=1.000000, inner sep=0pt, outer sep=0pt,scale=2]
  \path[draw=black,fill=ccccccc,line join=miter,line cap=butt,even odd rule,line
    width=0.800pt] (-24.9259,-39.9663) .. controls (-10.0099,-39.7633) and
    (5.8309,-23.3238) .. (0.0000,0.0000) .. controls (-2.0000,8.0000) and
    (-4.9107,16.3738) .. (-10.0000,22.3622) .. controls (-15.0893,28.3506) and
    (-22.3573,31.9537) .. (-29.1922,30.4615);

\draw[->,>=latex] (115,-35) -- node[right=5pt,midway]{$\xi_{t-1}$} (115,-15);
\draw[->,>=latex] (115,-15) node {$\bullet$}
-- (110,10) node {$\bullet$} node[midway,right=5pt]{$\xi_t$};
\draw[->,>=latex] (115,-15) to[bend left=10] node[midway,above=5pt]
{$\nabla h^*$} (-5,-10) node {$\bullet$} node[left=5pt] {$x_t$};
\draw[->,>=latex] (110,10) to[bend left=10] node[midway,above=5pt]
{$\nabla h^*$}  (-2,8) node {$\bullet$} node[left=5pt] {$x_{t+1}$};

\draw[dashed] (75,-50) -- (75,30);
\draw (40,-45) node {primal space};
\draw (110,-45) node {dual space};
\draw (-20,20) node {$\X$};
\end{tikzpicture}

  \caption{Dual averaging}
  \label{fig:lazy-mirror-descent}
\end{figure}


\section{The unified mirror descent algorithm}
\label{sec:unif-mirr-desc}
In this section, we introduce a general family of algorithms which we
refer to as {\em unified mirror descent (UMD)} and show that MD\ and DA\ are 
special cases.

\subsection{Definition, properties and special cases}
\label{sec:defin-prop-spec}

\begin{definition}
\label{def:definition-UMD}
Let $h$ be an $\X$-regularizer and
$\xi:= (\xi_t)_{t\geqslant 1}$ be a sequence in $\rset^n$. We say that
$(x_t,\vt_t)_{t\geqslant 1}$ is sequence of UMD iterates associated with
$h$ and $\xi$ (or a UMD$(h,\xi)$ sequence for short) if
for all $t\geqslant 1$:
\begin{subequations}
\begin{align}
\label{eq:umd-gradient-h-star}
  &x_{t}=\nabla h^*(\vt_{t}),\\
  \label{eq:umd-variational}
&\forall x\in \X,\ \left< \vt_{t+1}-\vt_{t}-\xi_{t} \middle| x-x_{t+1} \right>\geqslant 0.
\end{align}
\end{subequations}
Points $(x_t)_{t\geqslant 1}$ (resp.\ $(\vt_t)_{t\geqslant 1}$) are
called \emph{primal iterates} (resp.\ \emph{dual iterates}), and
vectors $(\xi_t)_{t\geqslant 1}$ are called \emph{dual increments}.
\end{definition}
\begin{proposition}
  \label{prop:umd-properties}
  Let $(x_t,\vt_t)_{t\geqslant 1}$ be an UMD($h,\xi$) sequence
  defined as above. Then for all $t\geqslant 1$,
\begin{enumerate}[(i)]
\item\label{item:y-in-partial-h-x} $\vt_t\in \partial h(x_t)$,
\item
\label{item:y-u-in-partial-h-x} $\vt_t+\xi_t\in \partial
  h(x_{t+1})$ and $x_{t+1}=\nabla h^*(\vt_t+\xi_t)$.
\end{enumerate}
\end{proposition}
\begin{proof}
  Let $t\geqslant 1$. By definition of UMD\ iterates,
  $x_t=\nabla h^*(\vt_t)$, which combined with
  Proposition~\ref{prop:subdifferential-inverse} of Appendix~\ref{sec:conv-analys-tools} implies (\ref{item:y-in-partial-h-x}).
Furthermore, for all $x\in \X$ we deduce from $\vt_{t+1}\in \partial
  h(x_{t+1})$ that
  \[ h(x)-h(x_{t+1})\geqslant \left< \vt_{t+1} \middle| x-x_{t+1} \right>\geqslant \left< \vt_t+\xi_t \middle| x-x_{t+1} \right>, \]
  where the second inequality results from \eqref{eq:umd-variational} in the  definition of UMD\ iterates. On the other hand,  $h(x)=+\infty$ for $x\notin \X$
implying the inequality $h(x)-h(x_{t+1})\geqslant \ps{\vt_t+\xi_t}{x-x_{t+1}}$ in this case. We conclude that  $\vt_t+\xi_t$ also belongs to $\partial h(x_{t+1})$, i.e.,
  (\ref{item:y-u-in-partial-h-x}) holds true.\qed
\end{proof}

\begin{figure}
  \centering

\begin{tikzpicture}[y=0.80pt, x=0.80pt, yscale=-1.000000, xscale=1.000000, inner sep=0pt, outer sep=0pt,scale=2]
  \path[draw=black,fill=ccccccc,line join=miter,line cap=butt,even odd rule,line
    width=0.800pt] (-24.9259,-39.9663) .. controls (-10.0099,-39.7633) and
    (5.8309,-23.3238) .. (0.0000,0.0000) .. controls (-4.0000,16.0000) and
    (-15.5225,33.4460) .. (-29.1922,30.4615);
\draw (0,0) node {$\bullet$};
\draw[->,>=latex] (-15,-15) node {$\bullet$} node[above=5pt]{$x_t$}
to[bend right=20] node[midway,below=5pt]{$\partial h$} (115,-30) node
{$\bullet$} node[right=5pt] {$\vt_t$};
\draw[->,>=latex] (115,-30) to (110,-5) node
{$\bullet$} node[right=5pt] {$\vt_t+\xi_t$};
\draw[->,>=latex] (110,-5) to [bend left=20]
node[below=5pt,midway]{$\nabla h^*$} (0,0) node{$\bullet$} node[left=5pt] {$x_{t+1}$};
\draw[->,>=latex] (0,0) to[bend left=20] node[midway,below=5pt]
{$\partial h$} (110,10) node{$\bullet$} node[right=5pt]{$\vt_{t+1}$};

\draw[dashed] (75,-50) -- (75,30);
\draw (40,-45) node {primal space};
\draw (110,-45) node {dual space};
\draw (-20,20) node {$\X$};
\end{tikzpicture}
  \caption{Unified mirror descent}
  \label{fig:unified-mirror-descent}
\end{figure}

\begin{remark}[Existence of UMD iterates]
As soon as $\X$-regularizer $h$ and sequence of dual
increments $(\xi_t)_{t\geqslant 1}$ are given, we can see that
UMD$(h,\xi)$ iterates always exist.  Indeed, from the definition of
regularizers, it follows that there exists a primal point
$x_1\in \X$ such that $\partial h(x_1)\neq \emptyset$; in other
words, there exists $(x_1,\vt_1)$ such that $x_1=\nabla h^*(\vt_1)$.
Then, for $t\geqslant 1$, one can consider $\vt_{t+1}:=\vt_t+\xi_t$
which indeed satisfies variational condition
\eqref{eq:umd-variational}, and then define
$x_{t+1}:=\nabla h^*(\vt _{t+1})$.  This choice of $\vt_{t+1}$ actually
corresponds to the iteration of the DA\ algorithm.
\end{remark}
\begin{proposition}[DA is a special case of UMD]
  \label{prop:da-special-case-umd}
  Let $h$ be an $\X$-regularizer,
  $\xi:= (\xi_t)_{t\geqslant 1}$ be a sequence in $\rset^n$. Let
  $(x_t,\vt_t)_{t\geqslant 1}$ be DA$(h,\xi)$ iterates.
  Then, $(x_t,\vt_t)_{t\geqslant 1}$ are UMD$(h,\xi)$ iterates.
\end{proposition}
\begin{proof}
  First, condition \eqref{eq:umd-gradient-h-star} is true by definition of \eqref{eq:lmd}.  Besides, the relation $\vt_{t+1}=\vt_t+\xi_t$ makes condition
\eqref{eq:umd-variational} trivially satisfied because one of the
arguments of the scalar product is zero.\qed
\end{proof}

\begin{proposition}[MD is a special case of UMD]
  \label{prop:md-special-case-umd}
Let $F$ be an $\X$-compatible mirror map and
$\xi:= (\xi_t)_{t\geqslant 1}$ be a sequence in $\rset^n$.
Let $(x_t)_{t\geqslant 1}$ be a sequence of MD$(F,\X,\xi)$ iterates.
Then, $(x_t,\nabla F(x_t))_{t\geqslant 1}$ is a sequence of UMD($F + I_{\X},\xi$) iterates.
\end{proposition}
\begin{proof}
  For $t\geqslant 1$, we consider $\vt_t:=\nabla F(x_t)$. Let us prove
  that conditions \eqref{eq:umd-gradient-h-star} and
  \eqref{eq:umd-variational} are satisfied with
  $h:=F+I_{\X}$.

For $t\geqslant 1$, denote $\tilde{x}_t:=\nabla F^*(\nabla
F(x_t)+\xi_t)$, which implies $\nabla F(\tilde{x}_t)=\nabla F(x_t)+\xi_t=\vt_t+\xi_t$ thanks to Proposition~\ref{prop:subdifferential-inverse} of the appendix.
We can then rewrite the \eqref{eq:gmd} iteration as follows:
\begin{align}
x_{t+1}&=\argmin_{x\in \X}D_F(x,\ \tilde{x}_t)\nonumber\\
&=\argmin_{x\in \X}\left\{ F(x)-F(\tilde{x}_t)-\left< \nabla
                                                            F(\tilde{x}_t)
                                                            \middle|
                                                            x-\tilde{x}_t\right>
                                                            \right\}\nonumber\\
&=\argmin_{x\in \X}\left\{ F(x)-\left< \vt_t+\xi_t
                                                                         \middle| x \right>  \right\}.
\label{eq:added-num-1}
\end{align}
In other words, $x_{t+1}$ is the minimizer on $\X$ of the
convex function $F(x)-\left< \vt_t+\xi_t \middle| x \right>$.
This function is differentiable at $x_{t+1}$ because we know by Theorem~\ref{thm:bregman-projection} that $x_{t+1}\in \mathcal{D}_F:=\operatorname{int}\operatorname{dom}F$
 and $F$ is differentiable on $\mathcal{D}_F$, so that the optimality conditions for \eqref{eq:added-num-1} imply that
 \be \forall x\in \X,\quad \left< \nabla F(x_{t+1})-\vt_t-\xi_t \middle| x-x_{t+1} \right>\geqslant 0,
 \ee{eq:similar-ineq}
 which is exactly condition \eqref{eq:umd-variational} due to $\vt_{t+1}=\nabla F(x_{t+1})$.

 Thanks to Proposition~\ref{prop:F+I_X}, we know that
 $\vt_t=\nabla F(x_t)\in \partial h(x_t)$ which is equivalent to
 $x_t\in \nabla h^*(\vt_t)$ (see
 Proposition~\ref{prop:subdifferential-inverse}); thus, condition
 \eqref{eq:umd-gradient-h-star} is satisfied.\qed
\end{proof}

One may consider the following alternative definition of UMD\ iterates.
Let $\prox_h:\rset^n\rightrightarrows \X\times \rset^n$
be a multi-valued \emph{prox-mapping} defined as follows. $\prox_h
(\zeta)$ is the set of couples
$(x,\vt)$ satisfying:
\begin{align*}
&x=\nabla h^*(\zeta)\\
&\vt\in \partial h(x)\\
&\forall x'\in \X,\ \ps{\vt-\zeta}{x'-x}\geqslant 0.
\end{align*}
Then, it can be easily checked that  $(x_t,\vt_t)_{t\geqslant 1}$ is a
sequence of UMD$(h,\xi)$ iterates if and only if:
\begin{align*}
  \begin{split}
\vt_1&\in \partial h(x_1),\\
(x_{t+1},\vt_{t+1}) &\in\prox_h(\vt_{t}+\xi_t),\quad t\geqslant 1.
  \end{split}
\end{align*}

\begin{remark}[On the non-unicity of UMD iterates]
An interesting characteristic of UMD\ iterates is that
for a given sequence $(\xi_t)_{t\geqslant 1}$ of dual increments and
initial points $(x_1,\vt_1)$, there may be several possible UMD sequences
because the prox-mapping $\prox_h$ is multi-valued.
However, as soon as the subdifferential $\partial h(x)$ is at most a
singleton at each point $x\in \X$, the prox-mapping $\prox_h$ is
single-valued and the UMD sequence is thus unique; in particular,  in such case, DA and
MD coincide. This is the case, for instance, if $\X=\rset^n$ and
the regularizer $h$ is differentiable on $\rset^n$.
\end{remark}
\subsection{Simple examples}
\label{sec:simple-examples}
As an illustration, let us describe the iterates of MD, DA and UMD
in the Euclidean setting corresponding to the
$\X$-regularizer $h(x)=\frac{1}{2}\left\| x
\right\|_2^2+I_{\X}$ and the mirror map
$F(x)=\frac{1}{2}\left\| x \right\|_2^2$.
It is easy to check that the map $\nabla h^*$ is the Euclidean
projection onto $\X$.
We consider below two simple cases of the set $\X$.
We denote $(x_t,\vt_t)_{t\geqslant 1}$ a sequence of UMD$(h,\xi)$ iterates.
\paragraph{Euclidean ball.}
Here we consider the case of $\X=\bar{B}(0,1)$,
the closed unit Euclidean ball of $\rset^2$. Let $t\geqslant 1$ and assume that
$\vt_{t}+\xi_{t}$ is outside $\bar{B}(0,1)$, so that $x_{t+1}$ which is
the Euclidean projection of $\vt_{t}+\xi_{t}$ belongs to the boundary
of $\bar{B}(0,1)$; thus, $\left\| x_{t+1} \right\|_2=1$. From
this point $x_{t+1}$, an MD iteration corresponds to choosing
$\vt_{t+1}^{\text{MD}}:=x_{t+1}$, and a DA iteration
corresponds to choosing $\vt_{t+1}^{\text{DA}}:=\vt_{t}+\xi_{t}$.
Besides this, we can see that the set of points $\vt_{t+1}$ which have $x_{t+1}$ as
Euclidean projection is
$[1,+\infty)\, x_{t+1}$ and
that the set of points $\vt_{t+1}$ satisfying condition
\eqref{eq:umd-variational} is $(-\infty,1]\, (\vt_{t}+\xi_{t})$.
Therefore, the set of vectors $\vt_{t+1}$ satisfying both conditions
\eqref{eq:umd-variational} and \eqref{eq:umd-gradient-h-star} is
the convex hull of $x_{t+1}$ and $\vt_{t}+\xi_{t}$, which is represented
by a thick segment in Figure~\ref{fig:euclidean-ball-example}.

\begin{figure}[h]
  \centering
  \begin{tikzpicture}[scale=2]
    \draw[fill=gray!20] (0,0) circle (1);
    \draw[ultra thick] (0,0) node{$\bullet$} node[below]{$0$} node[above=15pt]{$\X=\bar{B}(0,1)$}
          (1,0) node{$\bullet$} node[below right]{$x_{t+1}$}
          -- (2.5,0) node{$\bullet$} node[below right]{$\vt_{t}+\xi_{t}$};
    \node (xt) at (1,0) {};
    \node (yu) at (2.5,0) {};
    \draw[<-,bend right,>=latex] (xt) to (1.25,.5)
    node[above]{$\vt_{t+1}^{\text{MD}}$};
    \draw[<-,bend right,>=latex] (yu) to (2.75,.5) node[above]{$\vt_{t+1}^{\text{DA}}$};
  \end{tikzpicture}
  \caption{MD, DA, and \UMD\ iterations when
    $\X$ is the Euclidean ball.}
  \label{fig:euclidean-ball-example}
\end{figure}

\paragraph{Simplex in $\rset^2$}
Suppose that the ambient space is $\rset^2$,
and  \[\X=\{(x^{(1)},x^{(2)})\in \rset^2_+:\,  x^{(1)}+x^{(2)}\leq 1\}\]
is the ``full simplex'' of $\rset^2$. Let $t\geqslant 1$ and assume that $\eta=\vt_{t}+\xi_{t}$ belongs to the normal cone $\N_\X(\bar x)$ of $\X$ at $\bar x=(0,1)$, i.e.,
\[\N_\X(\bar x)=\{\vt\in \rset^2:\,\vt^{(2)}\geqslant 1,\,\vt^{(2)}-\vt^{(1)}\geqslant 1\},
\]
cf. Figure \ref{fig:simplex-example}. In this case, $x_{t+1}=\bar x$ and the set of points $\vt_{t+1}$ which have
$\bar x$ as Euclidean projection is $\N_\X(\bar x)$, while the set of $\vt_{t+1}$ satisfying \eqref{eq:umd-variational} is
\[
\mathcal{S}=\{\vt\in  \rset^2:\, \vt^{(2)}\leq \eta_t^{(2)},\, \vt^{(2)}-\vt^{(1)}\leqslant \eta^{(2)}-\eta^{(1)}\}.
\]
The set of vectors satisfying both
\eqref{eq:umd-variational} and \eqref{eq:umd-gradient-h-star} is represented by the dashed area in Figure~\ref{fig:simplex-example}.
\begin{figure}
\centering
\includegraphics[width=0.5\textwidth]{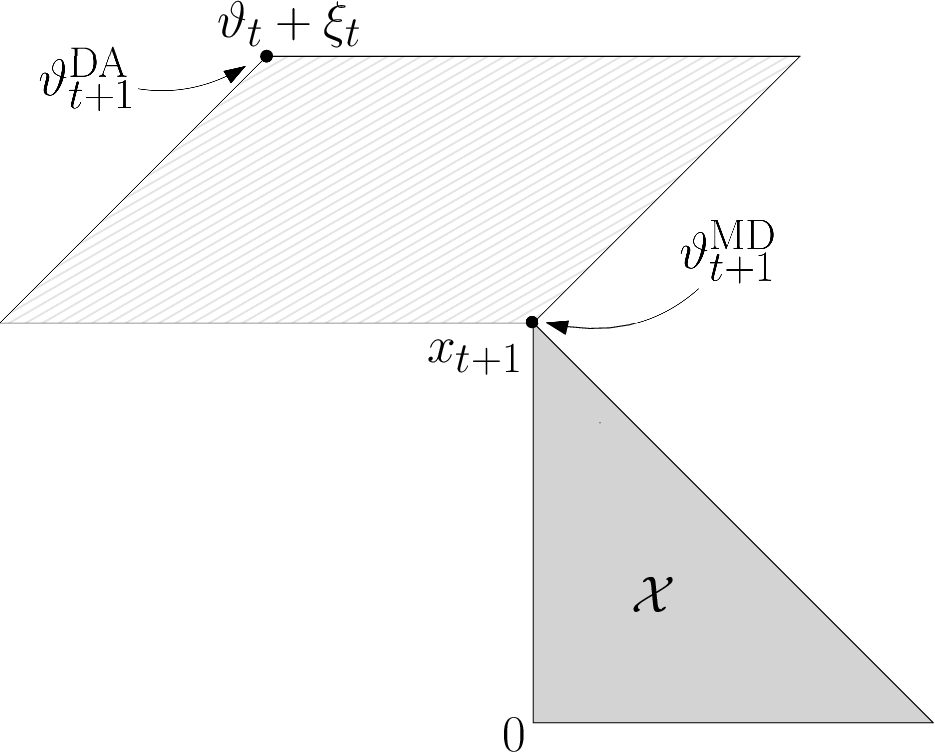}
\caption{\label{fig:simplex-example}Comparison of MD, DA, and UMD\ iterations when
  $\X$ is a ``full simplex'' in $\rset^2$.}
\end{figure}
\subsection{A class of iterates interpolating MD and DA}
\label{sec:class-iter-interp}

As an example which goes beyond MD and DA, we consider a class of UMD iterates which, at each step,
interpolates between a MD and a DA iteration; the IPDD algorithm
considered in Section~\ref{sec:igal-algorithm} is from this class.

Let $F$ be an $\mathcal{X}$-compatible mirror map, and
$h=F+I_{\mathcal{X}}$ the associated regularizer.
Let $\xi:=(\xi_t)_{t\geqslant 1}$ be a sequence of dual increments in
$\rset^n$, and $(\alpha_t)_{t\geqslant 1}$ be a sequence of coefficients in $[0,1]$.
Consider a sequence of iterates $(x_t,\vartheta_t)_{t\geqslant 1}$ as follows: $\vartheta_1\in \rset^n$, $x_1=\nabla h^*(\vartheta_1)$, and for $t\geqslant 2$
\begin{align}
\begin{split}
x_{t+1}&=\nabla h^*(\vartheta_t+\xi_t), \\
 \vartheta_{t+1}&=\alpha_t\nabla F(x_t)+(1-\alpha_t)(\vartheta_t+\xi_t).
 \end{split}
 \label{eq:interp-1}
 \end{align}
Note that selecting constant weights $(\alpha_t)_{t\geqslant 1}$ equal to
 $0$ (resp.\ equal to $1$) corresponds to DA (resp. MD).

\begin{proposition}
Sequence $(x_t,\vt_t)_{t\geqslant 1}$ as defined in \eqref{eq:interp-1} is a sequence
of UMD$(h,\xi)$ iterates.
\end{proposition}
\begin{proof}
$x_1=\nabla h^*(\vartheta_1)$ by definition. For
$t\geqslant 2$, establishing $x_t=\nabla h^*(\vartheta_t)$ is equivalent
(by Proposition~\ref{prop:subdifferential-inverse} of Appendix~\ref{sec:conv-analys-tools}) to proving that $\vartheta_t\in \partial h(x_t)$. By
definition, $\vartheta_t$ is a convex combination of $\nabla F(x_t)$
and $\vartheta_t+\xi_t$ which are both elements of a convex set $\partial h(x_t)$. Therefore, relationship \eqref{eq:umd-gradient-h-star} from Definition~\ref{def:definition-UMD} holds true.
On the other hand, for $x\in \mathcal{X}$ and $t\geqslant 1$,
\begin{align*}
\ps{\vartheta_{t+1}-\vartheta_t-\xi_t}{x-x_{t+1}}&=\ps{\alpha_t\nabla F(x_t)+(1-\alpha_t)(\vartheta_t+\xi_t)-\vartheta_t-\xi_t}{x-x_{t+1}}\\
&=\alpha_t\ps{\nabla F(x_t)-\vartheta_t-\xi_t}{x-x_{t+1}}\geqslant 0
\end{align*}
(the last inequality holds for the same reasons as the similar inequality \eqref{eq:similar-ineq} in the proof of Proposition~\ref{prop:md-special-case-umd}), implying variational condition \eqref{eq:umd-variational}.\qed
\end{proof}

\subsection{Analysis}
\label{sec:gener-bregm-diverg}
We introduce a natural extension of the Bregman divergence which will be central when analyzing the properties of UMD iterates.

\begin{definition}[Bregman divergence]
  \label{def:bregman}
  Let $g:\rset^n\to \rset\cup\{+\infty\}$ be a convex function. For
  $x\in \rset^n$ such that $\partial g(x)\neq \emptyset$,
  $x'\in \rset^n$, and $\vt\in \partial g(x)$, we define the Bregman
  divergence from $x$ to $x'$ with subgradient $\vt$ as
\[ D_g(x',x;\ \vt):=g(x')-g(x)-\left< \vt \middle| x'-x \right>. \]
\end{definition}
\begin{remark}
If $g$ is differentiable at  $x$,
the \emph{traditional} Bregman divergence from $x$ to $x'$ is well-defined and is equal to the Bregman divergence (as defined above) from $x$ to $x'$ with (the only) subgradient $\nabla g(x)$; in other words: $D_g(x',x;\ \nabla g(x))=D_g(x',x)$.

Note that a different generalization of the Bregman divergence using directional
derivatives was proposed in
\cite{dasgupta2012agglomerative,mcmahan2017survey}; it does not lead however
to a unifying view of mirror descent and dual averaging algorithms, due to the uniqueness of directional
derivatives.

\end{remark}
\begin{proposition}
  \label{prop:g-bregman-properties}
Let $g:\rset^n\to \rset\cup\{+\infty\}$ be a
lower-semicontinuous convex function.
Let $x,x',\vt,\vt'\in \rset^n$ be such that $\vt\in \partial g(x)$ and
$\vt'\in \partial g(x')$.
\begin{enumerate}[(i)]
\item\label{item:bregman-conjugate} Then,
\[ 0\leqslant D_g(x',x;\ \vt)=D_{g^*}(\vt,\vt';\ x'), \]
where $D_{g^*}$ is the Bregman
divergence associated with the Fenchel--Legendre transform $g^*$ of $g$.
\item\label{item:g-bregman-strong-convexity} Moreover, if $g$ is strongly convex (with modulus 1)\footnote{With some terminological abuse, we say that $g$ is strongly convex when it is strongly convex with modulus $1$.}
with respect to a
given norm $\left\|\,\cdot\,\right\|_{}$, $g^*$ is differentiable on
$\rset^n$ and
\[ \half\left\| x'-x \right\|^2\leqslant  D_g(x',x;\ \vt)=D_{g^*}(\vt,\vt')\leqslant \half\left\| \vt-\vt' \right\|_*^2. \]
\end{enumerate}
\end{proposition}
\begin{proof}
\begin{enumerate}[(i)]
\item The nonnegativity of $D_g$ is an immediate consequence of the convexity of $g$.
  Using the Fenchel identity (property (\ref{item:fenchel-identity})
  from Proposition~\ref{prop:subdifferential-inverse}), we write
\begin{align*}
D_g(x',x;\ \vt)&=g(x')-g(x)-\left< \vt \middle| x'-x \right>\\
&=\left< \vt' \middle| x' \right> -g^*(\vt')-\left< \vt \middle| x \right>
                                                              +g^*(\vt)-\left<
                                                              \vt\middle|
                                                              x'-x
                                                              \right>\\
&=g^*(\vt)-g^*(\vt')-\left< \vt-\vt' \middle| x' \right> =D_{g^*}(\vt,\vt';\ x').
\end{align*}
\item The differentiability of $g^*$ and the second inequality is
  given by \cite[Lemma~15]{shalev2007online}. For the first
  inequality, we refer to \cite[Lemma~13]{shalev2007online}.\qed
\end{enumerate}
\end{proof}

We now establish the following fundamental result,
which is an extension of classical statements \cite[Lemma 3]{chen1993convergence} and \cite[Lemma 4]{nesterov2007dual}.
It underlies the analysis of the algorithms of the mirror descent type and is operational when deriving accuracy guarantees in various applications of the UMD\ presented below.

 \begin{lemma}
   \label{lm:3-point}
   Let $h$ be an $\X$-regularizer, $\xi:=(\xi_t)_{t\geqslant
     1}$ be a sequence in $\rset^n$, and $(x_t,\vt_t)_{t\geqslant 1}$ a
   sequence of UMD$(h,\xi)$ iterates. Then, for all $x\in
   \operatorname{dom}h$ and $t\geqslant 1$,
\be
\ps{\xi_t}{x-x_{t+1}}\leqslant D_h(x,x_t;\ \vt_t)-D_h(x,x_{t+1};\
                       \vt_{t+1})-D_h(x_{t+1},x_t;\ \vt_t).
\ee{eq:3-point}
As a consequence,
\be
\ps{\xi_t}{x-x_t}\leqslant D_h(x,x_t;\ \vt_t)-D_h(x,x_{t+1};\ \vt_{t+1})+D_{h^*}(\vt_t+\xi_t,\vt_t).
\ee{eq:3-point-regret}
 \end{lemma}
\begin{proof}
Let $x\in \operatorname{dom}h$ and $t\geqslant 1$. Using variational inequality
\eqref{eq:umd-variational} from the definition of UMD iterates
we write
\begin{align*}
\left< \xi_t \middle| x-x_{t+1} \right>&\leqslant \left< \vt_{t+1}-\vt_t \middle| x-x_{t+1} \right>\\
&=\left< \vt_{t+1} \middle| x-x_{t+1} \right> -\left< \vt_t \middle| x-x_t \right> +\left< \vt_t \middle| x_{t+1}-x_t \right>\\
                                       &=\left( h(x)-h(x_t)-\left< \vt_t \middle| x-x_t \right>  \right)\\
  &\qquad\qquad  -\left( h(x)-h(x_{t+1})-\left< \vt_{t+1} \middle| x-x_{t+1} \right>  \right)\\
  &\qquad\qquad - \left( h(x_{t+1})-h(x_t)-\left< \vt_t \middle| x_{t+1}-x_t \right>   \right)\\
    &=D_h(x,x_t;\ \vt_t)-D_h(x,x_{t+1};\ \vt_{t+1})-D_h(x_{t+1},x_t;\ \vt_t).
\end{align*}
The above  divergences are indeed well-defined
because $\vt_t\in \partial h(x_t)$ and $\vt_{t+1}\in \partial h(x_{t+1})$
as a consequence of the definition of UMD iterates (property
(\ref{item:y-in-partial-h-x}) from Proposition~\ref{prop:umd-properties}).

To prove \eqref{eq:3-point-regret}, we note that
\[ \left< \xi_t \middle| x_{t+1}-x_t \right>  = D_h(x_{t+1},x_t;\ \vt_t) +D_h(x_t,x_{t+1};\ \vt_{t}+\xi_t), \]
where the second Bregman divergence is well-defined
because $\vt_t+\xi_t\in \partial h(x_{t+1})$ according to property
(\ref{item:y-u-in-partial-h-x}) from
Proposition~\ref{prop:umd-properties}.
Moreover,
\[ D_h(x_t,x_{t+1};\ \vt_{t}+\xi_t)=D_{h^*}(\vt_t+\xi_t,\vt_t;\
x_t)=D_{h^*}(\vt_t+\xi_t,\vt_t), \]
where the first equality is due to
Proposition~\ref{prop:g-bregman-properties}--(\ref{item:bregman-conjugate})
and the second equality---to the differentiability of $h^*$. Combining
the two previous displays and adding to \eqref{eq:3-point} gives the result.\qed
\end{proof}
\begin{remark}
The ``high level idea'' underlying the construction of the unified mirror descent is to combine elements of DA and MD algorithms.
MD makes use of a mirror map, which is
differentiable on the interior of its domain. As shown in
Figure~\ref{fig:greedy-mirror-descent}, its gradient
$\nabla F$ is used to go from the primal space to the dual space,
where the \emph{dual iteration} is performed. Then, $\nabla F^*$ is used to
come back to the primal space, where the Bregman projection is
done.

Our first observation is that this \emph{back-and-forth} between the
primal and dual spaces can be extended to regularizers which may be
non-differentiable. Indeed, the corresponding regularizer
$h=F+I_{\mathcal{X}}$ can be seen as a ``limit'' of mirror
maps whose values outside of $\mathcal{X}$ are sent up to
$+\infty$; when the regularizer is not differentiable, the corresponding
``limit'' of the gradients corresponds to the subdifferential.
Therefore, a natural idea is to define dual UMD iterates using the
subdifferential whenever the gradient does not exist, as in the
diagram from Figure~\ref{fig:unified-mirror-descent}, meaning that
$\vartheta_t\in \partial h(x_t)$, with
$x_{t+1}=\nabla h^*(\vartheta_t+\xi_t)$, $t\geqslant 1$.
Note that both MD and DA satisfy these relationships.

However, the above recursion ``as is'' may not obey accuracy bounds
which hold for MD and DA recursions. To make this recursion
``interesting'' we need to constrain the choice of $\vartheta_t$ in
the subdifferential of $h(x_t)$. Our second observation is that in the
proof of accuracy bounds for MD and DA, results similar to
Lemma~\ref{lm:3-point} are operational (cf. e.g.
\cite[Lemma~3]{chen1993convergence} and \cite[Lemma~4]{nesterov2007dual}), with MD and DA satisfying inequality \eqref{eq:3-point} of Lemma~\ref{lm:3-point} ``by construction''. One easily check that the variational condition
\eqref{eq:umd-variational} is exactly the constraint needed for 
inequality \eqref{eq:3-point} to hold.
\end{remark}

An immediate consequence of Lemma~\ref{lm:3-point} and property
(\ref{item:g-bregman-strong-convexity}) of
Proposition~\ref{prop:g-bregman-properties} is the following
 inequality (sometimes called \emph{regret bound}), which
extends and unifies classical guarantees on MD\ and DA---cf., e.g.,
\cite[Lemma 2.1]{nemirovski2009robust}, \cite[Lemma 4]{nesterov2007dual}, \cite[Theorems 5.2 \& 5.4]{bubeck2012regret}, etc.
\begin{corollary}
\label{cor:regret-bound}
Let $h$ be an $\X$-regularizer which is assumed to be
strongly convex with respect to some norm
$\left\|\,\cdot\,\right\|$, and $\xi:=(\xi_t)_{t\geqslant 1}$ be a
sequence in $\rset^n$. Let $(x_t,\vt_t)_{t\geqslant 1}$ be a sequence
of UMD$(h,\xi)$ iterates. Then for $T\geqslant 1$ and $x\in \operatorname{dom}h$,
\begin{equation}
\label{eq:regret-bound-norm}
\sum_{t=1}^T\ps{\xi_t}{x-x_t}\leqslant D_h(x,x_1;\ \vt_1)-D_h(x,x_{T+1};\ \vt_{T+1})+\frac{1}{2}\sum_{t=1}^T\left\| \xi_t \right\|_{*}^2
\end{equation}
where $\left\|\,\cdot\,\right\|_*$ is the norm conjugate to $\left\|\,\cdot\,\right\|$.
\end{corollary}

\section{Accuracy bounds for convex optimization}
\label{sec:guar-conv-nonc}
Throughout this section,
let $\X$ be a convex and closed domain of $\rset^n$, $h$ an
$\mathcal{X}$-regularizer, $\left\|\,\cdot\,\right\|$ some norm in
$\rset^n$, and we assume that $h$ is strongly convex w.r.t. $\left\|\,\cdot\,\right\|$.

%
\subsection{UMD for nonsmooth convex optimization}
\label{sec:lipsch-conv-optim}
Let $M>0$, and let $f:\rset^n\to\mathbb{R}\cup\{+\infty\}$ be a convex
function with domain $\operatorname{dom}f\supset\X$,  which is
sub-differentiable on $\X$, with bounded subgradients:
\[ \forall x\in \X,\ \forall \xi\in \partial f(x) ,\quad \left\| \xi \right\|_{*}\leqslant M.   \]
We suppose that the optimization problem
\begin{equation}
\label{eq:6}
f_*=\min_{x\in \X}f(x)
\end{equation}
is solvable, and denote $x_*\in \X$ a minimizer. Let
$(\gamma_t)_{t\geqslant 1}$ be a sequence of positive \emph{step-sizes}.
We consider a sequence  $(x_t,\vt_t)_{t\geqslant 1}$ of UMD$(h,\xi)$
iterates associated with dual increments
$\xi:=(-\gamma_tf'(x_t))_{t\geqslant 1}$, where $f'(x_t) \in \partial f(x_t)$. Namely, $\vt_1\in
\partial h(x_1)$, and for $t\geqslant 1$,
\[ (x_{t+1},\vt_{t+1})\in \prox_h(\vt_t-\gamma_tf'(x_t)). \]

The following result provides accuracy estimates for approximate
solutions $\overline{x}_T$ by UMD after $T$ iterations, computed either as
\[ \overline{x}_T=\biggl({\sum_{t=1}^{T}\gamma_t}\biggr)^{-1}\sum_{t=1}^{T}\gamma_tx_t,\qquad \text{or }\qquad \overline{x}_T\in \Argmin_{t=1,\dots,T}f(x_t). \]
In particular, it recovers known guarantees for the mirror descent
\cite[Theorem 3.3.5]{nemirovsky1983problem}, \cite[Theorem 4.1]{beck2003mirror} and
dual averaging \cite[Theorem 1]{nesterov2009primal} algorithms. The following result is a straightforward consequence of Corollary \ref{cor:regret-bound}.
\begin{proposition}
  \label{prop:lipschitz}
  Suppose that $x_*\in \operatorname{dom}h$.
  Then for all $T\geqslant 1$,
  \[ f(\overline{x}_T)-f_*\leqslant\biggl({\sum_{t=1}^T\gamma_t}\biggr)^{-1}\left[{D_h(x_*,x_1;\
      \vt_1)+\frac{M^2}{2}\sum_{t=1}^T\gamma_t^2}\right]. \]
Let $\Omega_{\X}$ be an upper estimate of $\sqrt{2D_h(x_*,x_1;\
    \vt_1)}$.\footnote{In the case of compact $\X$ one can take
  $\Omega_X=\left[\max_{x\in \X}2D_h(x,x_1;\ \vt_1)\right]^{1/2}$. Note that in this case due to strong convexity of $D_h(\cdot,x_1,\vt_1)$ one has $\Omega_\X\geq \max_{x\in \X}\|x-x_1\|$.}
UMD algorithm with constant step-sizes
  \[ \gamma_t\equiv\gamma=\frac{\Omega_{\X}}{M\sqrt{T}},\qquad t\geqslant 1, \]
  satisfies:
  \[ f(\overline x_T)-f_*\leqslant \frac{\Omega_{\X}M}{\sqrt{T}}. \]
\end{proposition}

Next, following \cite{Naz18,nesterov2015quasi}, let us consider an
alternative algorithm whose iterates
$(x_t,y_t,\vt_t)_{t\geqslant 1}$ are defined as $x_1=y_1$, $\vt_1\in \partial h(x_1)$, and for
$t\geqslant 1$,
\begin{align}
\begin{split}
(x_{t+1},\vt_{t+1})&\in \prox(\vt_t-\gamma_tf'(y_t)),\\
y_{t+1}&=(1-\nu_t)y_t+\nu_tx_{t+1},
\label{eq:NNS1}
\end{split}
\end{align}
where  $\nu_t\in (0,1)$ is given by
\[ \nu_t=\gamma_{t+1}\biggl(  \sum_{s=1}^{t+1}\gamma_s\biggr)^{-1},
\qquad t\geqslant 1.  \]
Note that $(x_t,\vt_t)_{t\geqslant 1}$ are UMD$(h,\xi)$ iterates, here
$\xi:=(-\gamma_tf'(y_t))_{t\geqslant 1}$.

The following statement provides accuracy bounds for the last iterate $y_T$ of the recursion \eqref{eq:NNS1} and generalizes respective accuracy
bounds of \cite{Naz18,nesterov2015quasi}.
\begin{proposition}
  \label{prop:quasi-monotone}
  Suppose that $x_*\in \operatorname{dom}h$.
Then for $T\geqslant 1$,
  \[ f(y_T)-f_*\leqslant \biggl({\sum_{t=1}^T\gamma_t}\biggr)^{-1}\left[D_h(x_*,x_1;\ \vt_1)+\frac{M^2}{2}\sum_{t=1}^T\gamma_t^2\right]. \]
  In particular, for constant
  step-sizes
  \[ \gamma_t\equiv\gamma=\frac{\Omega_{\X}}{M\sqrt{T}},\qquad t\geqslant 1, \]
  where $\Omega_{\X}$ is an upper bound for of $\sqrt{2D_h(x_*,x_1;\
    \vt_1)}$, one has
  \[ f(y_T)-f_*\leqslant \frac{\Omega_{\X}M}{\sqrt{T}}. \]
\end{proposition}
\subsection{UMD for smooth convex optimization}
\label{sec:umd-smooth-convex}
In this section, in the context of smooth convex optimization, we first present a class
of algorithms which generalizes gradient descent and enjoys a $1/T$
convergence rate. The new algorithms defined below in
Section~\ref{sec:gold:-new-algorithm} belong to this class. Moreover, we also
define a generalization of Nesterov's accelerated gradient method
which guarantees a $1/T^2$ convergence rate.

Let $f:\rset^n\to\mathbb{R}\cup\{+\infty\}$  be a convex function which is continuously differentiable on $\X$ with Lipschitz-continuous gradient, i.e.,
\be \left\| \nabla f(x)-\nabla f(x') \right\|_*\leqslant L\left\| x-x' \right\| ,\qquad x,x'\in \X. \ee{eq:lip-grad}
We assume that the problem
\[f_*=\min_{x\in \X} f(x)
\] is solvable and denote $x_*\in\mathcal{X}$ a minimizer.
Let $(\gamma_t)_{t\geqslant 1}$ be a sequence of positive step-sizes, and $(x_t,\vt_t)_{t\geqslant 1}$
be UMD$(h,\xi)$ iterates with
$\xi:=(-\gamma_t\nabla f(x_t))_{t\geqslant 1}$. In other words, $\vt_1\in
\partial h(x_1)$, and
\[ x_{t+1}=\prox_h(\vt_t-\gamma\nabla f(x_t)),\qquad t\geqslant 1. \]
\begin{theorem}
\label{thm:smooth-convex}
Suppose that $x_*\in \operatorname{dom}h$. Assume that step-sizes $(\gamma_t)_{t\geqslant 1}$ are chosen in such a way that condition
\be
\gamma_t D_f(x_{t+1},x_t)\leqslant D_h(x_{t+1},x_t;\ \vt_t)
\ee{eq:back}
is satisfied for all $t\geq 1$, which is always the case for $\gamma_t\leq 1/L$.
Then
for $T\geqslant 1$, one has
\begin{equation}
\label{eq:smooth_0}
f(x_{T+1})-f_*\leqslant \biggl(\sum_{t=1}^T\gamma_t\biggr)^{-1}D_h(x_*,x_1;\ \vt_1).
\end{equation}
In particular, for constant step-sizes $\gamma_t\equiv\gamma=1/L$, one has
 \[ f(x_{T+1})-f_*\leqslant{LD_h(x_*,x_1;\ \vt_1)\over T}. \]
\end{theorem}

We now aim at presenting an ``UMD analogue'' of Nesterov's accelerated
gradient descent algorithm, which unifies and generalizes classic
algorithmic schemes for smooth optimization, like e.g.\ optimal scheme for smooth minimization from \cite{nesterov2005smooth}, optimal method from \cite{lan2012optimal}, etc.

Let points $(x_t,y_t,z_t,\vt_t)_{t\geqslant 1}$
 satisfy $x_1=z_1=\nabla h^*(\vt_1)$, and for $t\geqslant 1$,
\begin{subequations}
\begin{align}
\label{eq:acc-y}
&y_{t}=(1-\nu_{t})z_{t}+\nu_{t}x_{t}\\
\label{eq:acc-prox}
&(x_{t+1},\vt_{t+1})\in \prox_h(\vt_t-\gamma_t\nabla f(y_t))\\
\label{eq:acc-z}
&\displaystyle z_{t+1}=(1-\nu_t)z_t+\nu_tx_{t+1},
\end{align}
\end{subequations}
with positive step-sizes $(\gamma_t)_{t\geqslant 1}$ and
$(\nu_t)_{t\geqslant 1}$ a sequence in $(0,1)$.

We refer to $(x_t,y_t,z_t,\vt_t)_{t\geqslant 1}$ as {\em accelerated unified mirror descent (AUMD)} iterates.
Note that AUMD iterates are well defined. Indeed, it follows from the
above definition that $(x_t,\vt_t)_{t\geqslant 1}$ is a sequence of
UMD$(h,\xi)$ iterates (associated with dual increments
$\xi:=(-\gamma_t\nabla f(y_t))_{t\geqslant 1}$). Therefore, $f$ being
differentiable on $\mathcal{X}$, $x_{t+1}$ do exist
whenever $y_t\in \X$. To show the latter, $y_t$ being a convex
combination of $x_t$ and $z_t$, it suffices to check that $x_t,z_t$
are well-defined and belong to $\X$, which can be done
recursively. Indeed, $x_1,z_1\in \X$ by construction. Then, assuming
that $x_t,z_t\in \X$, we get that $y_t\in \X$. As a result, $x_{t+1}$
is well-defined and belongs to $\X$, and so does $z_{t+1}$, being a
convex combination of $z_t$ and $x_{t+1}$.

The following result states the accuracy guarantees for the AUMD algorithm \eqref{eq:acc-y}--\eqref{eq:acc-z} and generalizes corresponding statements from
\cite[Theorem 2]{lan2012optimal} and \cite[Theorem 4.4]{beck2009fast}.
\begin{theorem}
 \label{thm:acceleration}
 Suppose that $x_*\in \operatorname{dom}h$.
 Let $\gamma_1=1/L$ and for $t\geqslant 1$,
\be\gamma_{t+1}=(2L)^{-1}\left(
  1+\sqrt{1+(2L \gamma_{t})^2} \right),\qquad \nu_t=(L \gamma_t)^{-1}. \ee{eq:nest-steps}
 Then for  $T\geqslant 1$, it holds:
 \be f(z_{T+1})-f_*\leqslant \frac{4L  D_h(x_*,x_1;\ \vt_1)}{(T+1)^2}. \ee{eq:nest-bound}
\end{theorem}

\section{APDD~and IPDD: new algorithms for constrained optimization}
\label{sec:gold:-new-algorithm}
In constrained convex
optimization problems with minimizer lying at the boundary of the
feasible domain, (primal) mirror descent algorithm often shows fast convergence,
but is sensitive to the selected step-size, whereas dual averaging
achieves slower convergence, but is robust to the step-size choice (cf. the  discussion below).

We introduce two new algorithms from the UMD family which we refer to
as \emph{Alternating Primal-Dual Descent (APDD)} and
\emph{Interpolating Primal-Dual Descent (IPDD)}, and which belong to
the class described in Section~\ref{sec:class-iter-interp}. They,
therefore, benefit from the theoretical guarantees from
Section~\ref{sec:guar-conv-nonc}. The idea behind their construction
is to alternate between different iterations types in order to combine
the advantages of mirror descent and dual averaging so that they
perform better than (or as well as) the best of both algorithms with
moderate computational overhead.\footnote{APDD~and
  IPDD~algorithms should be seen as mere examples having nothing
  special which sets them apart from other possible UMD
  implementations.}
In the reminder of this section, $f: \rset^d\to \rset$ is a differentiable function.

\subsection{The APDD algorithm}
\label{sec:gold-algorithms}

Every $k$ steps (where $k\geqslant 1$ is a parameter) the APDD
algorithm computes both a mirror descent update and a dual averaging
update, and then compares the value of the objective function at the
two primal iterates thus obtained. The algorithm then records the
so-determined best update and performs for the remaining $k-1$ steps
dual averaging updates. In a variant of the algorithm, the best update is determined over
not one, but over several steps ahead.

\begin{definition}[$k$-APDD algorithm]
  \label{def:gold}
  Let $k\geqslant 1$ be an integer, $\gamma>0$,
$F$ be an $\mathcal{X}$-compatible mirror map, $h=F+I_{\mathcal{X}}$,
  and $(x_1,\vt_1)\in
  \mathcal{X}\times \rset^d$ such that $x_1=\nabla h^*(\vt_1)$.

  The primal and dual iterates $(x_t,\vt_t)_{t\geqslant 1}$ of the $k$-APDD algorithm with step-size $\gamma$ and
  objective function $f$ are defined as $x_2=\nabla
  h^*(\vt_1-\gamma\nabla f(x_1))$ and for $t\geqslant 2$,
\begin{itemize}
\item if $t\equiv 2 \mod k$, $\vt_t\in \Argmin_{\vt\in \left\{ \vt_t^{\text{MD}},\vt_t^{\text{DA}} \right\}}f\left( \nabla h^*(\vt-\gamma\nabla f(x_t)) \right),  $
\\
where $\vt_t^{\text{MD}}=\nabla F(x_t)$ and $\vt_t^{\text{DA}}=\vt_{t-1}-\gamma\nabla f(x_{t-1})$;
\item if $t\not\equiv 2\mod k$,
$ \vt_t=
\vt_{t-1}-\gamma\nabla f(x_{t-1})
$.\item $x_{t+1}=\nabla h^*(\vt_t-\gamma\nabla f(x_t))$.
\end{itemize}
\end{definition}
Observe that numerical complexity of the iterate of the APDD does not exceed  twice the complexity of the iterate of the MD/DA-algorithm.

\subsection{The IPDD algorithm}
\label{sec:igal-algorithm}

The IPDD algorithm performs at time $t\geqslant 1$ an iteration which interpolates DA and MD
(with coefficient $\alpha$) if the condition
\be\gamma D_f(x_{t+1},x_t)\leqslant D_h(x_{t+1},x_t;\ \vt_t)\ee{eq:conv-cond}
is
satisfied at this iteration.\footnote{Recall, that satisfaction of such condition (cf. \eqref{eq:back}) at each iteration of the method, ensures that the bound \eqref{eq:smooth_0} of Theorem \ref{thm:smooth-convex} holds true.}  When this condition is not satisfied, the IPDD
algorithm performs a DA iteration instead.

\begin{definition}[The $\alpha$-IPDD algorithm]
  \label{def:igal}
  Let $\alpha\in (0,1]$, $\gamma>0$;
let also $F$ be a $\mathcal{X}$-compatible mirror map, $h=F+I_{\mathcal{X}}$,
  and $(x_1,\vt_1)\in
  \mathcal{X}\times \rset^d$ such that $x_1=\nabla h^*(\vt_1)$.

  The primal and dual iterates $(x_t,\vt_t)_{t\geqslant 1}$ of the $\alpha$-IPDD algorithm with step-size $\gamma$ and
  objective function $f$ are defined as $x_2=\nabla
  h^*(\vt_1-\gamma\nabla f(x_1))$ and for $t\geqslant 2$,
\begin{align*}
  \vt_t^{(0)}&=\alpha\nabla F(x_t)+(1-\alpha)(\vt_{t-1}-\gamma\nabla f(x_{t-1}))\\
  x_{t+1}^{(0)}&=\nabla h^*(\vt_t^{(0)}-\gamma\nabla f(x_t))\\
\vt_t&=
       \begin{cases}
        \vt_t^{(0)},&\text{if
          $\gamma D_f(x_{t+1}^{(0)},x_t)\leqslant D_h(x_{t+1}^{(0)},x_t; \vt_t^{(0)})$} \\
        \vt_{t-1}-\gamma\nabla f(x_{t-1}),&\text{otherwise},
       \end{cases}\\
  x_{t+1}&=\nabla h^*(\vt_t-\gamma\nabla f(x_{t})).
\end{align*}
\end{definition}
Note that at a given step, the IPDD
 algorithm first computes an iteration which interpolates DA and
 MD if condition \eqref{eq:conv-cond} is satisfied; additionally, a DA iteration is computed if condition \eqref{eq:conv-cond} is not satisfied.
 Thus, in the worst case, the algorithm computes two UMD iterations per step.

\begin{remark}
Clearly, the use of the APDD and IPDD algorithms makes no sense in situations
where mirror descent and dual averaging produce identical iterates,
which is the case, in particular, when the problem is unconstrained.  Even in a constrained case, when the algorithm iterates belong to the interior of $\X$ (e.g.,
when the problem solution is an interior point of $\X$), the mirror descent and dual
averaging update will end up coinciding and the APDD and IPDD
algorithms will no longer provide any improvement. However, the situation changes dramatically when the iterates of the method lie at the boundary of the feasible set.

From now on, let us consider a convex optimization problem with minimizer  at
the boundary of the feasible set $\X$. For simplicity, we discuss the
Euclidean instances of the algorithms, meaning that the associated mirror
map (resp.\ regularizer) is $F=\frac{1}{2}\left\|\,\cdot\,\right\|_2^2$
(resp. $h=F+I_{\mathcal{X}}$).

As the algorithm iterates reach the boundary of the feasible set, possible UMD iterates
differ. Informally, with dual increments $\xi_t=-\gamma \nabla f(x_t)$ in the normal cone of $\X$ at $x_t$, the dual iterate $\vt_t$ of DA (the iterate ``before the projection'') gets farther and
farther away from $\mathcal{X}$. By contrast, the
dual iterate $\vartheta_t$ of MD always belongs to the feasible set, by
definition.
Consequently, successive primal DA iterates becomes closer to each
other and more conservative, compared to the MD iterates which vary more
aggressively. As a result, MD with large
step-sizes produces iterates which vary too much (thus compromising
convergence), whereas variations of DA iterates with large
step-sizes are quickly attenuated exactly because the step-sizes are large (and dual iterates are ``far from $\X$''), implying that dual iterates $\vartheta_t$
are getting away from $\X$ even faster. This difference is clearly observed in
the numerical experiments below.

An alternative intuition on the robustness of DA to the choice
of step-sizes is as follows. As stated in Theorem~\ref{thm:smooth-convex}, in the case of the smooth objective function, accuracy bounds for the UMD still hold
if condition
\eqref{eq:conv-cond} 
is satisfied for all $t\geqslant 1$.  Observe that
\[ D_h(x_{t+1},x_t;\ \theta_t)=D_F(x_{t+1},x_t)+\ps{\nabla F(x_t)-\vartheta_t}{x_{t+1}-x_t}. \]
In the Euclidean case, when $\nabla F(x_t)=x_t$ and $x_{t+1}$ are in
$\X$, the farther $\vartheta_t$ from $\mathcal{X}$, the
larger the above quantity.  As $\vartheta_t$ gets farther and
farther away from $\mathcal{X}$, the ratio
$D_h(x_{t+1},x_t;\ \vartheta_t)/D_f(x_{t+1},x_t)$ becomes large,
and condition \eqref{eq:conv-cond} for DA iterates is satisfied with a large step-size $\gamma$.

The $k$-APDD algorithm mostly performs DA iterations, and MD iterations once
in a while. Each MD iteration brings the dual point $\vt_t$ back to
the feasible set $\mathcal{X}$, which has the effect of making subsequent iterations
more aggressive. When the step-size is too large, these iterations may lead to a temporary increase of the objective value,
but the following DA iterations push the dual point $\vt_t$
away from $\mathcal{X}$, recovering proper convergence in most cases.

The $\alpha$-IPDD algorithm performs a ``moderately aggressive
iteration'' most of the time which is an interpolation between DA and MD iterations, with
coefficient $\alpha\in(0,1]$. When such interpolation is
too aggressive for maintaining proper convergence, a DA iteration is
performed instead, pushing the dual point $\vt_t$ further away from
$\mathcal{X}$, thus making subsequent iterations more conservative. Occasionally, aggressive iterations may lead to a
temporary increase of the objective value, which is always less significant than in the case of APDD.
In the numerical experiments we report on below, we choose $\alpha=.1$
which seems to be a good trade-off between aggressiveness (for achieving
fast convergence) and conservativeness (for maintaining the dual point
$\vt_t$ in a region were convergence is sustainable).
\end{remark}

\subsection{Numerical experiments}
To illustrate the numerical performance of the proposed algorithms, we present here results of a preliminary computational experiment involving comparison of APDD and IPDD with existing algorithms in two optimization problems arising in the statistical treatment of large datasets.
\subsubsection{Least-squares regression}
\label{sec:numer-exper-with}
We consider a least-squares regression problem using the training
sample of the BlogFeedback dataset from the UCI Machine Learning
Repository\footnote{\url{https://archive.ics.uci.edu/ml/datasets/BlogFeedback}} with rescaled---multiplied by $0.005$---features (regressors). The corresponding least-squares problem with dimensions $n=52397$ and $d=280$ writes
\begin{align*}
  \min_{x\in \X}\left\{f(x)=\frac{1}{n}\sum_{i=1}^n(b_i-\ps{x}{a_i})^2\right\}
\end{align*}
where \[
\X=\left\{ x\in \rset^d:\, \| x \|_2 \leqslant 1 \right\}\]
The spectral norm of the matrix $Q={1\over n}\sum_{i=1}^na_ia_i^T$ 
(the Lipschitz constant of the gradient of the objective function) is $ \lambda_{\max}(Q)=571.34$; the matrix is ill conditioned with only 20 eigenvalues exceeding $10^{-6}\lambda_{\max}(Q)$.  A high accuracy approximation of the optimal value $f_*$ (with optimal solution at the boundary of the feasible set) is first computed with a long run of
projected gradient descent with manually tuned step-size.

We compute approximate solutions by the following algorithms: mirror descent (MD), dual averaging (DA), 1-APDD,
5-APDD, and $.1$-IPDD. We consider the Euclidean setting which
corresponds to the mirror map
$F=\half\left\| \,\cdot\, \right\|_2^2$ and regularizer
$h=F+I_{\mathcal{X}}$. Iterations are initialized with $\vt_1=0$ and
$x_1=\nabla h^*(\vt_1)=0$, we run $T=200$ steps of each methods with the constant step-size.
The results of the experiment are presented in Figure \ref{fig:plots} where we plot the evolution of the suboptimality gap $f(x_t)-f_*$ for several values of the step-sizes. When the step-size is less than its ``theoretically justified'' value $\gamma_*=(\lambda_{\max}(Q))^{-1}=0.0018$ all algorithms converge slowly.  For $\gamma=\gamma_*$, all algorithms except DA converge linearly, with MD and 1-APDD achieving
the fastest convergence. MD does not
converge for $\gamma\geq 10\gamma_*$, DA converges sublinearly, and the APDD and the IPDD
algorithms exhibit linear convergence (with APDD showing temporal increase in the objective value); convergence of the IPDD algorithm seems to be unaffected by the large value of the step-size.

 \begin{figure}
\caption{Least-squares regression: evolution of suboptimality of approximate solution by mirror descent, dual
  averaging, 1-APDD, 5-APDD, and IPDD algorithms for different values of $\gamma$.}
\label{fig:plots}
   \centering
 \footnotesize
 \begin{subfigure}{.5\textwidth}
 \begin{scaletikzpicturetowidth}{\textwidth}
 \input{plots3e-08.tex}
 \end{scaletikzpicturetowidth}
 \caption{$\gamma=\gamma_*$}
 \end{subfigure}~\begin{subfigure}{.5\textwidth}
 \begin{scaletikzpicturetowidth}{\textwidth}
 \input{plots1e-07.tex}
 \end{scaletikzpicturetowidth}
 \caption{$\gamma=2.5\gamma_*$}
 \end{subfigure}
 \medskip

 \begin{subfigure}{.5\textwidth}
 \begin{scaletikzpicturetowidth}{\textwidth}
 \input{plots1e-06.tex}
 \end{scaletikzpicturetowidth}
 \caption{$\gamma=25\gamma_*$}
 \end{subfigure}~
 \begin{subfigure}{.5\textwidth}
 \begin{scaletikzpicturetowidth}{\textwidth}
 \input{plots1.tex}
 \end{scaletikzpicturetowidth}
 \caption{$\gamma=10^7\gamma_*$}
 \end{subfigure}
 \end{figure}
\subsubsection{Logistic regression}
\label{sec:numer-exper-with-1}
We consider a logistic regression problem using the training
sample of the Madelon dataset from the UCI Machine Learning
Repository\footnote{\url{https://archive.ics.uci.edu/ml/datasets/Madelon}}
with all features multiplied by $10^{-3}$.
The corresponding minimization problem of dimensions $n=2000$ and $d=500$ writes
\[
  \min_{x\in \X}\left\{f(x)=\frac{1}{n}\sum_{i=1}^n\log \left(1+e^{-b_i \ps{x}{a_i}} \right)\right\}
\]
with the same as in the preceding example feasible set
\[ \X=\left\{ x\in \rset^d:\, \left\| x \right\|_2 \leqslant 1 \right\}. \]
In the present case, the spectral norm of the matrix $Q={1\over n}\sum_{i=1}^na_ia_i^T$ 
(the Lipschitz constant of the gradient of the objective function) is $ \lambda_{\max}(Q)=119.16$.   An estimate of the  optimal value $f_*$ (with optimal solution at the boundary of the feasible set) is first computed with a long run of
projected gradient descent with manually tuned step-size.

We consider the following algorithms: mirror descent (MD), dual averaging (DA),
20-APDD, 20-7-APDD,\footnote{The second parameter in the definition of the $k$-$\ell$-ADPP corresponds to the $\ell$-step ahead computation of the objective when determining the choice of update every $k$ steps of the algorithm.} and $.1$-IPDD.
As in Section~\ref{sec:numer-exper-with}, we consider the Euclidean setup of the problem,
with the mirror map $F=\half\left\| \,\cdot\, \right\|_2^2$ and regularizer
$h=F+I_{\mathcal{X}}$. Iterations are initialized with $\vt_1=0$ and
$x_1=\nabla h^*(\vt_1)=0$. We run $T=5000$ steps of each algorithm for several values of constant step-sizes; the results are presented in Figure \ref{fig:plots2}.

In view of the problem parameters, the ``theoretically justified'' choice of the step-size is $\gamma_*=(\lambda_{\max}(Q))^{-1}=0.0084$.
In our experiments, we observe linear convergence of all algorithms except for DA when the algorithm step-size $\gamma<0.07(\approx 8\gamma_*)$. MD does not converge for  $\gamma=0.07$, DA converges
sublinearly, the remaining algorithms achieving the same
linear convergence. When $\gamma\geq .1$, 20-APDD and 20-7-APDD algorithms start to oscillate without converging, while DA
converges sublinearly, and IPDD continues exhibiting linear convergence. We observe yet
another regime for larger step-sizes, e.g.\ $\gamma=1$ or $\gamma=200$: MD does not converge, IPDD
converging linearly, 20-APDD still oscillating, but, surprisingly,
20-7-APDD enjoying the same sublinear convergence as DA.

The observed results may be summarized as follows: in our experiments, MD converges linearly only when run with a properly selected step-size,
while DA always converges sublinearly. The $k$-$\ell$-APDD algorithm is robust w.r.t. the step-size choice, it converges linearly in a wide range of step-sizes, and, same as DA, enjoys sublinear convergence for large step-sizes. The IPDD algorithm converges linearly and seems to be insensitive to the choice of step-size.

\begin{figure}
\caption{Logistic regression: evolution of suboptimality of approximate solution by mirror descent, dual
  averaging, 20-APDD, 20-7-APDD and $.1$-IPDD algorithms for different values of $\gamma$.}
\label{fig:plots2}
   \centering
 \footnotesize
 \begin{subfigure}{.5\textwidth}
 \begin{scaletikzpicturetowidth}{\textwidth}
 \input{plots20.06.tex}
 \end{scaletikzpicturetowidth}
 \caption{$\gamma=0.06$}
 \end{subfigure}~\begin{subfigure}{.5\textwidth}
 \begin{scaletikzpicturetowidth}{\textwidth}
 \input{plots20.07.tex}
 \end{scaletikzpicturetowidth}
 \caption{$\gamma=0.07$}
 \end{subfigure}
 \medskip

 \begin{subfigure}{.5\textwidth}
 \begin{scaletikzpicturetowidth}{\textwidth}
 \input{plots20.1.tex}
 \end{scaletikzpicturetowidth}
 \caption{$\gamma=0.1$}
 \end{subfigure}~\begin{subfigure}{.5\textwidth}
 \begin{scaletikzpicturetowidth}{\textwidth}
 \input{plots2200.tex}
 \end{scaletikzpicturetowidth}
 \caption{$\gamma=200$}
 \end{subfigure}
 \end{figure}

\section*{Acknowledgements}

The authors are grateful to Roberto Cominetti, Cristóbal Guzmán, Nicolas Flammarion and Sylvain Sorin for inspiring discussions and suggestions.  A.\ Juditsky was supported by MIAI {@} Grenoble Alpes (\texttt{ANR-19-P3IA-0003}). J.\ Kwon was supported by a public grant as part of the ``Investissement d'avenir'' project (\texttt{ANR-11-LABX-0056-LMH}), LabEx LMH.

\bibliographystyle{plain}
\bibliography{unifying-greedy-lazy}

\begin{thebibliography}{10}

\bibitem{audibert2009minimax}
Jean-Yves Audibert and S{\'e}bastien Bubeck.
\newblock Minimax policies for adversarial and stochastic bandits.
\newblock In {\em Proceedings of the 22nd Annual Conference on Learning Theory
  (COLT)}, pages 217--226, 2009.

\bibitem{audibert2010regret}
Jean-Yves Audibert and S{\'e}bastien Bubeck.
\newblock Regret bounds and minimax policies under partial monitoring.
\newblock {\em The Journal of Machine Learning Research}, 11:2785--2836, 2010.

\bibitem{audibert2013regret}
Jean-Yves Audibert, S{\'e}bastien Bubeck, and G{\'a}bor Lugosi.
\newblock Regret in online combinatorial optimization.
\newblock {\em Mathematics of Operations Research}, 39(1):31--45, 2013.

\bibitem{beck2003mirror}
Amir Beck and Marc Teboulle.
\newblock Mirror descent and nonlinear projected subgradient methods for convex
  optimization.
\newblock {\em Operations Research Letters}, 31(3):167--175, 2003.

\bibitem{beck2009fast}
Amir Beck and Marc Teboulle.
\newblock A fast iterative shrinkage-thresholding algorithm for linear inverse
  problems.
\newblock {\em SIAM journal on imaging sciences}, 2(1):183--202, 2009.

\bibitem{bregman1967relaxation}
Lev~M. Bregman.
\newblock The relaxation method of finding the common point of convex sets and
  its application to the solution of problems in convex programming.
\newblock {\em USSR Computational Mathematics and Mathematical Physics},
  7(3):200--217, 1967.

\bibitem{bubeck2011introduction}
S{\'e}bastien Bubeck.
\newblock {\em Introduction to Online Optimization: Lecture Notes}.
\newblock Princeton University, 2011.

\bibitem{bubeck2015convex}
S{\'e}bastien Bubeck.
\newblock Convex optimization: {A}lgorithms and complexity.
\newblock {\em Foundations and Trends in Machine Learning}, 8(3-4):231--357,
  2015.

\bibitem{bubeck2012regret}
S{\'e}bastien Bubeck and Nicol{\`o} Cesa-Bianchi.
\newblock Regret analysis of stochastic and nonstochastic multi-armed bandit
  problems.
\newblock {\em Machine Learning}, 5(1):1--122, 2012.

\bibitem{bubeck2012towards}
S{\'e}bastien Bubeck, Nicol{\`o} Cesa-Bianchi, and Sham~M. Kakade.
\newblock Towards minimax policies for online linear optimization with bandit
  feedback.
\newblock In {\em JMLR: Workshop and Conference Proceedings (COLT)}, volume~23,
  pages 41.1--41.14, 2012.

\bibitem{cesa2006prediction}
Nicol{\`o} Cesa-Bianchi and G{\'a}bor Lugosi.
\newblock {\em Prediction, Learning, and Games}.
\newblock Cambridge University Press, 2006.

\bibitem{chen1993convergence}
Gong Chen and Marc Teboulle.
\newblock Convergence analysis of a proximal-like minimization algorithm using
  {B}regman functions.
\newblock {\em SIAM Journal on Optimization}, 3(3):538--543, 1993.

\bibitem{cohen2017tight}
Alon Cohen, Tamir Hazan, and Tomer Koren.
\newblock Tight bounds for bandit combinatorial optimization.
\newblock {\em Proceedings of Machine Learning Research (COLT 2017)}, 65:1--14,
  2017.

\bibitem{cox2014dual}
Bruce Cox, Anatoli Juditsky, and Arkadi Nemirovski.
\newblock Dual subgradient algorithms for large-scale nonsmooth learning
  problems.
\newblock {\em Mathematical Programming}, 148(1-2):143--180, 2014.

\bibitem{dasgupta2012agglomerative}
Sanjoy Dasgupta and Matus~J. Telgarsky.
\newblock Agglomerative {B}regman clustering.
\newblock In {\em Proceedings of the 29th International Conference on Machine
  Learning (ICML 12)}, pages 1527--1534, 2012.

\bibitem{dekel2012optimal}
Ofer Dekel, Ran Gilad-Bachrach, Ohad Shamir, and Lin Xiao.
\newblock Optimal distributed online prediction using mini-batches.
\newblock {\em Journal of Machine Learning Research}, 13(Jan):165--202, 2012.

\bibitem{duchi2012dual}
John~C. Duchi, Alekh Agarwal, and Martin~J. Wainwright.
\newblock Dual averaging for distributed optimization: Convergence analysis and
  network scaling.
\newblock {\em IEEE Transactions on Automatic control}, 57(3):592--606, 2012.

\bibitem{duchi2016local}
John~C. Duchi and Feng Ruan.
\newblock Asymptotic optimality in stochastic optimization.
\newblock {\em The Annals of Statistics}, to appear.

\bibitem{flammarion2017stochastic}
Nicolas Flammarion and Francis Bach.
\newblock Stochastic composite least-squares regression with convergence rate
  {O}(1/n).
\newblock {\em Proceedings of Machine Learning Research (COLT 2017)}, 65:1--44,
  2017.

\bibitem{hazan2012convex}
Elad Hazan.
\newblock The convex optimization approach to regret minimization.
\newblock In S.~Nowozin S.~Sra and S.~Wrigh, editors, {\em Optimization for
  Machine Learning}, pages 287--303. MIT press, 2012.

\bibitem{juditsky2011_2}
Anatoli Juditsky and Arkadi Nemirovski.
\newblock First order methods for nonsmooth convex large-scale optimization,
  {II}: utilizing problems structure.
\newblock {\em Optimization for Machine Learning}, pages 149--183, 2011.

\bibitem{juditsky2008learning}
Anatoli Juditsky, Philippe Rigollet, and Alexandre~B. Tsybakov.
\newblock Learning by mirror averaging.
\newblock {\em The Annals of Statistics}, 36(5):2183--2206, 2008.

\bibitem{juditsky2005recursive}
Anatoli~B. Juditsky, Alexander~V. Nazin, Alexandre~B. Tsybakov, and Nicolas
  Vayatis.
\newblock Recursive aggregation of estimators by the mirror descent algorithm
  with averaging.
\newblock {\em Problems of Information Transmission}, 41(4):368--384, 2005.

\bibitem{lan2012optimal}
Guanghui Lan.
\newblock An optimal method for stochastic composite optimization.
\newblock {\em Mathematical Programming}, 133(1):365--397, 2012.

\bibitem{lee2012manifold}
Sangkyun Lee and Stephen~J. Wright.
\newblock Manifold identification in dual averaging for regularized stochastic
  online learning.
\newblock {\em Journal of Machine Learning Research}, 13(Jun):1705--1744, 2012.

\bibitem{mcmahan2011follow}
Brendan McMahan.
\newblock Follow-the-regularized-leader and mirror descent: Equivalence
  theorems and l1 regularization.
\newblock In {\em Proceedings of the Fourteenth International Conference on
  Artificial Intelligence and Statistics}, pages 525--533, 2011.

\bibitem{mcmahan2017survey}
H.~Brendan McMahan.
\newblock A survey of algorithms and analysis for adaptive online learning.
\newblock {\em The Journal of Machine Learning Research}, 18(1):3117--3166,
  2017.

\bibitem{Naz18}
Alexander~V. Nazin.
\newblock Algorithms of inertial mirror descent in convex problems of
  stochastic optimization.
\newblock {\em Autom. Remote Control}, 79(1):78--88, 2018.

\bibitem{nemirovski1979efficient}
Arkadi Nemirovski.
\newblock Efficient methods for large-scale convex optimization problems.
\newblock {\em Ekonomika i Matematicheskie Metody}, 15, 1979.

\bibitem{nemirovski2004prox}
Arkadi Nemirovski.
\newblock Prox-method with rate of convergence {O}(1/t) for variational
  inequalities with {L}ipschitz continuous monotone operators and smooth
  convex-concave saddle point problems.
\newblock {\em SIAM Journal on Optimization}, 15(1):229--251, 2004.

\bibitem{nemirovski2009robust}
Arkadi Nemirovski, Anatoli Juditsky, Guanghui Lan, and Alexander Shapiro.
\newblock Robust stochastic approximation approach to stochastic programming.
\newblock {\em SIAM Journal on Optimization}, 19(4):1574--1609, 2009.

\bibitem{nemirovsky1983problem}
Arkadi Nemirovski and David~B. Yudin.
\newblock {\em Problem Complexity and Method Efficiency in Optimization}.
\newblock Wiley Interscience, 1983.

\bibitem{nesterov2005smooth}
Yurii Nesterov.
\newblock Smooth minimization of non-smooth functions.
\newblock {\em Mathematical programming}, 103(1):127--152, 2005.

\bibitem{nesterov2007dual}
Yurii Nesterov.
\newblock Dual extrapolation and its applications to solving variational
  inequalities and related problems.
\newblock {\em Mathematical Programming}, 109(2-3):319--344, 2007.

\bibitem{nesterov2009primal}
Yurii Nesterov.
\newblock Primal-dual subgradient methods for convex problems.
\newblock {\em Mathematical programming}, 120(1):221--259, 2009.

\bibitem{nesterov2015quasi}
Yurii Nesterov and Vladimir Shikhman.
\newblock Quasi-monotone subgradient methods for nonsmooth convex minimization.
\newblock {\em Journal of Optimization Theory and Applications},
  165(3):917--940, 2015.

\bibitem{rakhlin2009lecture}
Alexander Rakhlin and Ambuj Tewari.
\newblock Lecture notes on online learning.
\newblock 2009.

\bibitem{rockafellar1970convex}
R.~Tyrrell Rockafellar.
\newblock {\em Convex Analysis}.
\newblock Princeton University Press, 1970.

\bibitem{shalev2007online}
Shai Shalev-Shwartz.
\newblock {\em Online learning: Theory, algorithms, and applications}.
\newblock PhD thesis, The Hebrew University of Jerusalem, 2007.

\bibitem{shalev2011online}
Shai Shalev-Shwartz.
\newblock Online learning and online convex optimization.
\newblock {\em Foundations and Trends in Machine Learning}, 4(2):107--194,
  2011.

\bibitem{zinkevich2003online}
Martin Zinkevich.
\newblock Online convex programming and generalized infinitesimal gradient
  ascent.
\newblock In {\em Proceedings of the Twentieth International Conference on
  Machine Learning (ICML)}, 2003.

\end{thebibliography}

\appendix

\section{Convex analysis tools}
\label{sec:conv-analys-tools}
\begin{definition}[Lower-semicontinuity]
A function $g:\rset^n\to \rset\cup\{+\infty\}$ is
lower-semicontinuous if for all $c\in \rset$, the sublevel set
$\left\{ x\in \rset^n\,\colon f(x)\leqslant c \right\}$ is closed.
\end{definition}
One can easily check that the sum of two lower-semicontinuous
functions is lower-semicontinuous. Continuous functions and characteristic
functions $I_{\X}$ of closed sets $\X\subset
\rset^n$ are examples of lower-semicontinuous functions.

\begin{definition}[Strong-convexity]
  \label{def:strong-convexity}
Let $g:\rset^n\to \rset\cup\{+\infty\}$,
$\left\|\,\cdot\,\right\|_{}$ be a norm in $\rset^n$ and $K > 0$.
Function $g$ is said to be strongly convex with modulus $\kappa$ with respect to norm
$\left\|\,\cdot\,\right\|_{}$ if for all $x,x'\in \rset^n$ and $\lambda\in \left[ 0,1 \right]$,
\[ g(\lambda x+(1-\lambda)x')\leqslant \lambda
  g(x)+(1-\lambda)g(x')-\frac{\kappa \lambda(1-\lambda)}{2}\left\| x'-x \right\|_{}^2. \]
\end{definition}



\begin{proposition}[Theorem~23.5 in \cite{rockafellar1970convex}]
\label{prop:subdifferential-inverse}
Let $g:\rset^n\to \rset\cup\{+\infty\}$ be a
lower-semicontinuous convex function with nonempty domain. Then for
all $x,y\in \rset^n$, the following statements are equivalent.
\begin{enumerate}[(i)]
\item $x\in \partial g^*(y)$;
\item $y\in \partial g(x)$;
\item\label{item:fenchel-identity} $\left< y \middle| x \right> =g(x)+g^*(y)$;
\item $x\in \Argmax_{x'\in \rset^n}\left\{ \left< y \middle| x' \right> - g(x')\right\}$;
\item $y\in \Argmax_{y'\in \rset^n}\left\{ \left< y' \middle| x \right> - g^*(y')\right\}$.
\end{enumerate}
\end{proposition}

\section{Postponed proofs}
\label{sec:postponed-proofs}
\subsection{Proofs for Section \ref{sec:MD-DA}}
\label{sec:proofs-2}
\subsubsection{Proof of Proposition~\ref{prop:mirror-map-properties}}
  Let $\vt\in \rset^n$. By property~(\ref{item:gradient-all-values})
  from Definition~\ref{def:mirror-maps},
  there exists $x_1\in \mathcal{D}_F$ such that $\nabla
  F(x_1)=\vt$. Therefore, function $\varphi_{\vt}:x\mapsto \left< \vt \middle|
    x \right> -F(x)$ is differentiable at $x_1$ and $\nabla
  \varphi_{\vt}(x_1)=0$. Moreover, $\varphi_{\vt}$ is strictly concave as a
  consequence of property~(\ref{item:F-lsc-strict-convex}) from
  Definition~\ref{def:mirror-maps}. Therefore, $x_1$ is the unique
  maximizer of $\varphi_{\vt}$ and:
  \[ F^*(\vt)=\max_{x\in \rset^n}\left\{ \left< \vt \middle| x \right> -F(x) \right\}<+\infty,  \]
  which proves property~(\ref{item:F-star-dom}).

  Besides, we have
\begin{equation}
\label{eq:1}
x_1\in \partial F^*(\vt) \quad \Longleftrightarrow \quad \vt=\nabla
    F(x_1) \quad \Longleftrightarrow \quad   \text{$x_1$ minimizer of $\phi_{\vt}$},
\end{equation}
  where the first equivalence comes from
  Proposition~\ref{prop:subdifferential-inverse}. Point $x_1$ being
  the unique maximizer of $\varphi_{\vt}$, we have that $\partial F^*(\vt)$
  is a singleton. In other words, $F^*$ is differentiable in $\vt$ and
\begin{equation}
\label{eq:2}
\nabla F^*(\vt)=x_1\in \mathcal{D}_F.
\end{equation}
  First, the above \eqref{eq:2} proves property~(\ref{item:F-star-diff}). Second, this equality combined with the equality from \eqref{eq:1} gives the second identity from property~(\ref{item:gradient-F-star-inverse}). Third, this proves that $\nabla F^*(\rset^n)\subset \mathcal{D}_F$.

  It remains to prove the reverse inclusion to get
  property~(\ref{item:F-star-image}). Let $x\in \mathcal{D}_F$. By
  property~(\ref{item:F-diff}) from Definition~\ref{def:mirror-maps},
  $F$ is differentiable in $x$. Consider
\begin{equation}
\label{eq:3}
\vt:=\nabla F(x),
\end{equation} and all the
  above holds with this special point $\vt$. In particular, $x_1=x$ by
  uniqueness of $x_1$. Therefore \eqref{eq:2} gives
\begin{equation}
\label{eq:4}
\nabla F^*(\vt)=x,
\end{equation}
  and this proves $\nabla F^*(\rset^n)\supset \mathcal{D}_F$ and
  thus property~(\ref{item:F-star-image}). Combining \eqref{eq:3} and
  \eqref{eq:4} gives the first identity from property~(\ref{item:gradient-F-star-inverse}).

\subsubsection{Proof of Theorem~\ref{thm:bregman-projection}}

   Let $x_0\in \mathcal{D}_F$. By definition of the mirror map, $F$ is
  differentiable at $x_0$. Therefore, $D_F(x,x_0)$ is well-defined for
  all $x\in \rset^n$.

  For all real value $\alpha\in \rset$,
  consider the sublevel set $S_{\X}(\alpha)$ of function $x\mapsto
  D_F(x,x_0)$ associated with value $\alpha$ and restricted to $\X$:
  \[ S_{\X}(\alpha):=\left\{ x\in \X\,\colon D_F(x,x_0)\leqslant \alpha \right\}. \]
    Inheriting properties from $F$, function $D_F(\,\cdot\,,x_0)$ is
  lower-semicontinuous and strictly convex: consequently, the
  sublevel sets $S_{\X}(\alpha)$ are closed and convex.

 Let us also prove that the sublevel sets $S_{\X}(\alpha)$ are
 bounded. For each value $\alpha\in \rset$, we write
 \[ S_{\X}(\alpha)\subset S_{\rset^n}(\alpha):=\left\{ x\in \rset^n\,\colon \,D_F(x,x_0)\leqslant \alpha \right\} \]
 and aim at proving that the latter set is bounded. By contradiction,
 let us suppose that there exists an unbounded sequence in
 $S_{\rset^n}(\alpha)$: let $(x_k)_{k\geqslant 1}$ be such that $0<\left\| x_k-x_0 \right\|_{}\xrightarrow[k \rightarrow
 +\infty]{}+\infty$ and $D_F(x_k,x_0)\leqslant \alpha$ for all
 $k\geqslant 1$. Using the Bolzano--Weierstrass theorem, there exists
 $v\neq 0$ and a subsequence $(x_{\phi(k)})_{k\geqslant 1}$ such that
 \[ \frac{x_{\phi(k)}-x_0}{\left\| x_{\phi(k)}-x_0 \right\|
   }\xrightarrow[k \rightarrow +\infty]{}v. \]
 The point $x_0+\frac{x_{\phi(k)}-x_0}{\left\| x_{\phi(k)}-x_0 \right\|}$ being a convex combination of $x_0$ and $x_{\phi(k)}$, we can write the
 corresponding convexity inequality for function $D_F(\,\cdot\,,x_0)$:
 \begin{align*}
D_F\left( x_0+\lambda_k(x_{\phi(k)}-x_0),x_0 \right)&\leqslant
   (1-\lambda_k)D_F(x_0,x_0) +\lambda_kD_F(x_{\phi(k)},x_0 )\\
&\leqslant \lambda_k\alpha \xrightarrow[k \rightarrow +\infty]{}0,
\end{align*}
where we used shorthand $\lambda_k:=\left\|  x_{\phi(k)}-x_0 \right\|^{-1}$. For the first above inequality, we used $D_F(x_0,x_0)=0$ and that
$D_F(x_{\phi(k)},x_0)\leqslant \alpha$ by definition of
$(x_k)_{k\geqslant 1}$. Then, using the lower-semicontinuity of
$D_F(\,\cdot\,,x_0)$ and the fact that
$x_0+\lambda_k(x_{\phi(k)}-x_0) \xrightarrow[k \rightarrow
+\infty]{}x_0+v$, we have
\[ D_F(x_0+v,x_0)\leqslant \liminf_{k\to +\infty}D_F(x_0+\lambda_k(x_{\phi(k)}-x_0),x_0)\leqslant \liminf_{k\to +\infty}\lambda_k\alpha=0. \]
The Bregman divergence of a convex function being nonnegative, the
above implies $D_F(x_0+v,x_0)=0$. Thus, function
$D_F(\,\cdot\,,x_0)$ attains its minimum ($0$) at two different points
(at $x_0$ and at $x_0+v$): this contradicts its strong convexity. Therefore, sublevel sets $S_{\X}(\alpha)$ are bounded and thus compact.

We now consider the value $\alpha_{\text{inf}}$ defined as
\[ \alpha_{\text{inf}}:=\inf\left\{
    \alpha\,\colon S_{\X}(\alpha)\neq \emptyset \right\}. \]
In other words, $\alpha_{\text{inf}}$ is the infimum value of
$D_F(\,\cdot\,,x_0)$ on $\X$, and thus the
only possible value for the minimum (if it exists).
We know that $\alpha_{\text{inf}} \geqslant 0$ because the Bregman
divergence is always nonnegative.
From the definition of the sets $S_{\X}(\alpha)$, it easily
follows that:
\[ S_{\X}(\alpha_{\text{inf}})=\bigcap_{\alpha > \alpha_{\text{inf}}}^{}S_{\X}(\alpha). \]
Naturally, the sets $S_{\X}(\alpha)$ are increasing in
$\alpha$ with respect to the inclusion order. Therefore,
$S_{\X}(\alpha_{\text{inf}})$ is the intersection of a
nested sequence of nonempty compact sets. It is thus nonempty
as well by Cantor's intersection theorem. Consequently,
$D_F(\,\cdot\,,x_0)$ does admit a minimum on $\X$, and the
minimizer is unique because of the strong convexity.

Let us now prove that the minimizer $x_*:=\argmin_{x\in \X}D_F(x,\ x_0)$ also belongs to $\mathcal{D}_F$.
Let us assume by contradiction that $x_*\in \X\setminus
\mathcal{D}_F$. By definition of the mirror map, $\X\cap
\mathcal{D}_F$ is nonempty; let $x_1\in \X\cap
\mathcal{D}_F$.
The set $\mathcal{D}_F$ being open by definition, there exists
$\varepsilon > 0$ such that the closed Euclidean ball
$\bar{B}(x_1,\varepsilon)$ centered in $x_1$ and of radius
$\varepsilon$ is a subset of $\mathcal{D}_F$.
We consider the convex hull
\[ \mathcal{C}:=\operatorname{co}\left( \left\{ x_* \right\}\cup
  \bar{B}(x_1,\varepsilon) \right), \] which is clearly is a compact set.

Consider function $G$ defined by:
\[ G(x):=D_F(x,x_0)=F(x)-F(x_0)-\ps{\nabla F(x_0)}{x-x_0}, \]
so that $x_*$
is the minimizer of $G$ on $\X$. In particular, $G$ is finite in
$x_*$.  $G$ inherits strict convexity, lower-semicontinuity, and differentiability on
$\mathcal{D}_F$ from function $F$. $G$ is continuous on the compact set
$\bar{B}(x_1,\varepsilon)$ because $G$ is convex on the open set
$\mathcal{D}_F\supset \bar{B}(x_1,\varepsilon)$.  Therefore, $G$ is
bounded on $\bar{B}(x_1,\varepsilon)$.  Let us prove that $G$ is also
bounded on $\mathcal{C}$. Let $x\in \mathcal{C}$. By definition of
$\mathcal{C}$, there exists $\lambda\in [0,1]$ and
$x'\in \bar{B}(x_1,\varepsilon)$ such that
$x=\lambda x_*+(1-\lambda)x'$.  By convexity of $G$, we have:
\[ G(x)\leqslant \lambda G(x_*)+(1-\lambda)G(x')\leqslant G(x_*)+G(x'). \] We know that
$G(x_*)$ is finite and that $G(x')$ is bounded for
$x'\in \bar{B}(x_1,\varepsilon)$.  Therefore $G$ is bounded on
$\mathcal{C}$: let us denote $G_{\text{max}}$ and $G_{\text{min}}$
some upper and lower bounds for the value of $G$ on $\mathcal{C}$.

Because $\X$ is a convex set, the segment
$[x_*,x_1]$ (in other words the convex hull of $\left\{ x_*,x_1
\right\}$) is a subset of $\X$. Besides, let us prove that
the set
\[ (x_*,x_1]:=\left\{ (1-\lambda)x_*+\lambda x_1\,\colon \lambda\in (0,1] \right\} \]
is a subset of $\mathcal{D}_F$. Let
$x_{\lambda}:=(1-\lambda)x_*+\lambda x_1$ (with $\lambda\in (0,1]$) a
point in the above set, and let us prove that it belongs to
$\mathcal{D}_F$. By definition of the mirror map, we have
$\X\subset \operatorname{cl}\mathcal{D}_F$, and
besides $x_*\in \X$ by definition. Therefore, there exists a sequence $(x_k)_{k\geqslant 1}$ in
$\mathcal{D}_F$ such that $x_k\to x_*$ as $k\to +\infty$. Then, we can
write
\begin{align*}
x_{\lambda}&=(1-\lambda)x_*+\lambda x_1\\
&=(1-\lambda)x_k + (1-\lambda)(x_*-x_k)+\lambda x_1\\
&=(1-\lambda)x_k+\lambda\left( x_1+\frac{1-\lambda}{\lambda}(x_*-x_k) \right).
\end{align*}
Since $x_k\to x_*$, for high enough $k$, the point
$x_1+(1-\lambda)\lambda^{-1}(x_*-x_k)$ belongs to
$\bar{B}(x_1,\varepsilon)$ and therefore to $\mathcal{D}_F$. Then,
the point $x_{\lambda}$ belongs to the convex set\footnote{The domain
  of a convex function is convex, and therefore
  $\mathcal{D}_F=\operatorname{int}\operatorname{dom}F$ is convex as
  the interior of a convex set.} $\mathcal{D}_F$ as
the convex combination of two points in $\mathcal{D}_F$. Therefore,
$(x_*,x_1]$ is indeed a subset of $\mathcal{D}_F$.

\begin{figure}
  \centering
  \label{fig:x_lambda}
  \begin{tikzpicture}[scale=4,rotate=270]
    \draw[thin,dotted] (0,0) node{$\bullet$} node[left]{$x_*$} --
    (0,1) node{$\bullet$} node[below]{$x_{\lambda}$} -- (0,2)
    node{$\bullet$} node[above right]{$x_1$};
    \draw[thick] (0,2) circle (.6);
    \draw[thick] (0,0) -- ({acos(.6/2)}:{sqrt(2^2-.6^2)});
    \draw[thick] (0,0) -- ({180-acos(.6/2)}:{sqrt(2^2-.6^2)});
    \draw {(0,2)+(135:.6)} node  {$\bullet$} node[above
    right]{$x_1+\varepsilon u$};

    \draw[<->,>=latex] {(0,2)+(45:.6)} -- node[midway,above right]{$\varepsilon$} (0,2);

    \draw[thin,dotted] {(0,1)+(135:.6/2)} node{$\bullet$}
    node[left]{$x_{\lambda}+\lambda\varepsilon u$}
-- node[midway]{$\bullet$} node[midway,left]{$x_{\lambda}+\lambda hu$} (0,1);
\draw[thin,dotted] {(0,2)+(135:.6)} -- (0,0);
  \end{tikzpicture}
\end{figure}
$G$ being differentiable on $\mathcal{D}_F$ by definition of the mirror
map, the gradient of $G$ exists at each point of $(x_*,x_1]$. Let us
prove that $\nabla G$ is bounded on $(x_*,x_1]$. Let $x_{\lambda}\in
(x_*,x_1]$, where $\lambda\in (0,1]$ is such that
\[ x_{\lambda}= (1-\lambda)x_*+\lambda x_1, \]
and let $u\in \rset^n$ such that $\left\| u \right\|_2=1$.
The point $x_1+\varepsilon u$ belongs to $\mathcal{C}$ because it
belongs to $\bar{B}(x_1,\varepsilon)$. The following point also
belongs to convex set $\mathcal{C}$ as the convex combination of $x_*$
and $x_1+\varepsilon u$ which both belong to $\mathcal{C}$:
\begin{equation}
\label{eq:5}
x_{\lambda}+\lambda \varepsilon u = (1-\lambda)x_*+\lambda(x_1+\varepsilon u)\in \mathcal{C}.
\end{equation}
Let $h\in (0,\varepsilon]$. The following point also belongs to
$\mathcal{C}$ as a convex combination of $x_{\lambda}$ and the above
point $x_{\lambda}+\lambda\varepsilon u$:
\begin{equation}
\label{eq:7}
x_{\lambda}+\lambda hu = \left( 1-\frac{h}{\varepsilon}
\right)x_{\lambda}+\frac{h}{\varepsilon}\left( x_{\lambda}+\lambda
  \varepsilon u \right)\in \mathcal{C}.
\end{equation}
Now using for $G$ the convexity inequality associated with the convex
combination from \eqref{eq:7}, we write:
\begin{align}
  \begin{split}
  \label{eq:8}
G(x_{\lambda}+h\lambda u)-G(x_{\lambda})&\leqslant
  \frac{h}{\varepsilon}\left( G(x_{\lambda}+\lambda\varepsilon
    u)-G(x_{\lambda}) \right)\\
  &=\frac{h}{\varepsilon}\left( G(x_{\lambda}+\lambda\varepsilon
    u)-G(x_*)+G(x_*)-G(x_{\lambda}) \right)\\
  &\leqslant \frac{h}{\varepsilon}\left(
    G(x_{\lambda}+\lambda\varepsilon u)-G(x_*) \right),
  \end{split}
\end{align}
where for the last line we used $G(x_*)\leqslant G(x_{\lambda})$ which
is true because $x_{\lambda}$ belongs to $\X$ and $x_*$ is by
definition the minimizer of $G$ on $\X$. Using the convexity
inequality associated with the convex combination from \eqref{eq:5},
we also write
\begin{align}
  \label{eq:9}
  \begin{split}
G(x_{\lambda}+\lambda\varepsilon u)-G(x_*)&\leqslant \lambda\left(
  G(x_1+\varepsilon u)-G(x_*) \right)\\
&\leqslant \lambda\left( G_{\text{max}}-G_{\text{min}} \right).
  \end{split}
\end{align}
Combining \eqref{eq:8} and \eqref{eq:9} and dividing by $h\lambda$, we
get
\[ \frac{G(x_{\lambda}+h\lambda u)-G(x_{\lambda})}{h\lambda}\leqslant \frac{G_{\text{max}}-G_{\text{min}}}{\varepsilon}. \]
Taking the limit as $h\to 0^+$, we get that $\left< \nabla
  G(x_{\lambda}) \middle| u \right> \leqslant
(G_{\text{max}}-G_{\text{min}})/\varepsilon$.
This being true for all vector $u$ such that $\left\| u \right\|_2=1$,
we have
\[ \left\| \nabla G(x_{\lambda}) \right\|_{2}=\max_{\left\| u \right\|_{2}=1 }\left< \nabla G(x_{\lambda}) \middle| u \right>\leqslant \frac{G_{\text{max}}-G_{\text{min}}}{\varepsilon}.  \]
As a result, $\nabla G$ is bounded on $(x_*,x_1]$.

Let us deduce that $\partial G(x_*)$ is nonempty.
The sequence $(\nabla G(x_{1/k}))_{k\geqslant 1}$ is bounded. Using
the Bolzano--Weierstrass theorem, there exists a subsequence $(\nabla
G(x_{1/\phi(k)}))_{k\geqslant 1}$ which converges to some vector
$\vt_*\in \rset^n$. For each $k\geqslant 1$, the following is
satisfied by convexity of $G$:
\[ \left< \nabla G(x_{1/\phi(k)}) \middle| x-x_{1/\phi(k)} \right>
  \leqslant G(x)-G(x_{1/\phi(k)}),\quad x\in \rset^n. \] Taking the
limsup on both sides for each $x\in \rset^n$ as $k\to +\infty$, we get
(because obviously $x_{1/\phi(k)}\to x_*$):
\[ \left< \vt_* \middle| x-x_* \right> \leqslant G(x)-\liminf_{k\to +\infty}G(x_{1/\phi(k)})\leqslant G(x)-G(x_*),\quad x\in \rset^n, \]
where the second inequality follows from the lower-semicontinuity of $G$. Consequently, $\vt_*$ belongs to $\partial G(x_*)$.

But by definition of the mirror map $\nabla F$ takes all possible
values and so does $\nabla G$, because it follows from the definition
of $G$ that $\nabla G=\nabla F-\nabla F(x_0)$. Therefore,
there exists a point $\tilde{x}\in \mathcal{D}_F$ (thus $\tilde{x}
\neq x_*$) such that $\nabla G(\tilde{x})=\vt_*$. Considering the point
$x_{\text{mid}}=\frac{1}{2}(x_*+\tilde{x})$, we can write the
following convexity inequalities:
\begin{align*}
\left< \vt_* \middle| x_{\text{mid}}-x_* \right> &\leqslant
                                                 G(x_{\text{mid}})-G(x_*)\\
\left< \vt_* \middle| x_{\text{mid}}-\tilde{x} \right> &\leqslant G(x_{\text{mid}})-G(\tilde{x}).
\end{align*}
We now add both inequalities and use the fact that
$x_{\text{mid}}-\tilde{x}=x_*-x_{\text{mid}}$ by definition of
$x_{\text{mid}}$ to get $0\leqslant
2G(x_{\text{mid}})-G(x_*)-G(\tilde{x})$, which can also be written
\[ G\left( \frac{x_*+\tilde{x}}{2} \right)\geqslant \frac{G(x_*)+G(\tilde{x})}{2}, \]
which contradicts the strong convexity of $G$. We conclude that $x_*\in \mathcal{D}_F$.

\subsubsection{Proof of Proposition~\ref{prop:sufficient-conditions-dom-h-star-R}}

Let $\vt\in \rset^n$. For each of the three assumptions, let us prove
that $h^*(\vt)$ is finite. This will prove that $\operatorname{dom}h^*=\rset^n$.
\begin{enumerate}[(i)]
\item Because $\operatorname{cl}\operatorname{dom}h=\X$ by definition
  of a pre-regularizer, we have:
  \[ h^*(\vt)=\max_{x\in \rset^n}\left\{ \ps{\vt}{x}-h(x)
    \right\}=\max_{x\in \X}\left\{ \ps{\vt}{x}-h(x) \right\}. \]
  Besides, the function $x\mapsto \ps{\vt}{x}-h(x)$ is
  upper-semicontinuous and therefore attains a maximum on $\X$ because
  $\X$ is assumed to be compact. Therefore $h^*(\vt)<+\infty$.
\item Because $\nabla h(\mathcal{D}_h)=\rset^n$ by assumption, there
  exists $x\in \mathcal{D}_h$ such that $\nabla h(x)=\vt$. Then, by
  Proposition~\ref{prop:subdifferential-inverse}, $h^*(\vt)=\ps{\vt}{x}-h(x)<+\infty$.
\item The function $x\mapsto \ps{\vt}{x}-h(x)$ is strongly concave on
  $\rset^n$ and therefore admits a maximum. Therefore, $h^*(\vt)<+\infty$.
\end{enumerate}

\subsubsection{Proof of Proposition~\ref{prop:differentiability-h-star}}

  Let $\vt\in \rset^n$.
Because $\operatorname{dom}h^*=\rset^n$, the subdifferential $\partial
h^*(\vt)$ is nonempty---see e.g.\ \cite[Theorem~23.4]{rockafellar1970convex}.
  By Proposition~\ref{prop:subdifferential-inverse}, $\partial
  h^*(\vt)$ is the set of maximizers of function $x\mapsto
  \ps{\vt}{x}-h(x)$, which is strictly concave. Therefore, the
  maximizer is unique and $h^*$ is differentiable at $\vt$.

  Let $x\in \mathcal{D}_F$ and let us prove that $\nabla F(x)\in \partial h(x)$.
By convexity of $F$, the following is true
\[ \forall x'\in \rset^n,\quad F(x')-F(x)\geqslant \left< \nabla F(x) \middle| x'-x \right>. \]
By definition of $h$, we obviously have $h(x')\geqslant F(x')$ for all
$x'\in \rset^n$, and $h(x)=F(x)+I_{\X}(x)=F(x)$
because $x\in \X$. Therefore, the following is also true
\[ \forall x'\in \rset^n,\quad h(x')-h(x)\geqslant \left< \nabla F(x) \middle| x'-x \right>. \]
In other words, $\nabla F(x)\in \partial f(x)$.

\subsubsection{Proof of Proposition~\ref{prop:F+I_X}}

$h$ is strictly convex as the sum of two convex functions, one of
which ($F$) is strictly convex. $h$ is lower-semicontinuous as the sum
of two lower-continuous functions.

Let us now prove that
$\operatorname{cl}\operatorname{dom}h=\X$. First, we write
\[ \operatorname{dom}h=\operatorname{dom}(F+I_{\X})=\operatorname{dom}F\cap \operatorname{dom}I_{\X}=\operatorname{dom}F\cap \X. \]
Let $x\in
\operatorname{cl}\operatorname{dom}h=\operatorname{cl}(\operatorname{dom}F\cap
\X)$. There exists a sequence $(x_k)_{k\geqslant 1}$ in
$\operatorname{dom}F\cap \X$ such that $x_k\to x$.
In particular, each $x_k$ belongs to closed set $\X$, and so does the limit: $x\in \X$.

Conversely, let $x\in \X$ and let us prove that $x\in
\operatorname{cl}(\operatorname{dom}F\cap \X)$ by
constructing a sequence $(x_k)_{k\geqslant 1}$ in $\operatorname{dom}F\cap
\X$ which converges to $x$. By definition of the mirror map,
we have $\X\subset \operatorname{cl}\mathcal{D}_F$, where $\mathcal{D}_F:=\operatorname{int}\operatorname{dom}F$. Therefore,
there exists a sequence $(x_l')_{l\geqslant 1}$ in $\mathcal{D}_F$ such
that $x_l'\to x$ as $l\to +\infty$. From the definition of the mirror map, we also have
that $\X\cap \mathcal{D}_F\neq \emptyset$. Let $x_0\in
\X\cap \mathcal{D}_F$. In particular, $x_0$ belongs
$\mathcal{D}_F$ which is an open set by definition. Therefore, there
exists a neighborhood $U\subset \mathcal{D}_F$ of point $x_0$. We now
construct the sequence $(x_k)_{k\geqslant 1}$ as follows:
\[ x_k:=\left( 1-\frac{1}{k} \right)x+\frac{1}{k}x_0,\quad k\geqslant 1. \]
$x_k$ belongs to $\X$ as the convex combination of two points
in the convex set $\X$, and obviously converges to $x$. Besides, $x_k$ can also be written, for any $k,l\geqslant 1$,
\begin{align*}
x_k&= \left( 1-\frac{1}{k} \right)x'_l+\left( 1-\frac{1}{k} \right)(x-x_l')+\frac{1}{k}x_0 \\
  &=\left( 1-\frac{1}{k} \right)x_l'+ \frac{1}{k}\left(
     x_0+(k-1)(x-x_l') \right)\\
&=\left( 1-\frac{1}{k} \right)x_l'+\frac{1}{k}x_{0}^{(kl)},
\end{align*}
where we set $x_0^{(kl)}:=x_0+(k-1)(x-x_l')$. For a given $k\geqslant
1$, we see that $x_0^{(kl)}\to x_0$ as $l\to +\infty$ because $x_l'\to
x$ by definition of $(x_l')_{l\geqslant 1}$. Therefore, for
large enough $l$, $x_0^{(kl)}$ belongs to the neighborhood $U$ and
therefore to $\mathcal{D}_F$. $x_k$ then appears as the convex
combination of $x_l'$ and $x_0^{(kl)}$ which both belong to the convex
set $\mathcal{D}_F\subset \operatorname{dom}F$. $(x_k)$ is thus
a sequence in $\operatorname{dom}F\cap \X$ which converges
to $x$. Therefore, $x\in \operatorname{cl}(\operatorname{dom}F\cap
\X)$ and $h$ is an $\X$-pre-regularizer.

Finally, we have $F\leqslant h$ by definition of $h$. One can easily
check that this implies $h^*\leqslant F^*$ and we know from
Proposition~\ref{prop:mirror-map-properties} that
$\operatorname{dom}F^*=\rset^n$, in other words that $F^*$ only
takes finite values. Therefore, so does $h^*$ and $h$ is an $\X$-regularizer.

\subsection{Proofs for Section \ref{sec:guar-conv-nonc}}

\subsubsection{Proof of Proposition \ref{prop:quasi-monotone}}
  Let $t\geqslant 2$.
It follows from  the definition of the iterates that
$x_t-y_t=(\nu_{t-1}^{-1}-1)(y_t-y_{t-1})$. Therefore, utilizing the
convexity of $f$, we get
\begin{align*}
\left< \gamma_tf'(y_t) \middle| x_t-x_* \right>
  &=\gamma_t\left< f'(y_t) \middle| y_t-x_* \right>
    +\gamma_t\left< f'(y_t) \middle| x_t-y_t \right>\\
&=\gamma_t\left< f'(y_t) \middle| y_t-x_* \right>
                                                                  +\gamma_t(\nu_{t-1}^{-1}-1)\left< f'(y_t) \middle| y_t-y_{t-1} \right>\\
&\geqslant \gamma_t\left( f(y_t)-f_*
                                                                                                                                                       \right)+\gamma_t(\nu_{t-1}^{-1}-1)\left(
                                                                                                                                                       f(y_t)-f(y_{t-1})
                                                                                                                                                       \right)
  \\
  &= \gamma_t\nu_{t-1}^{-1}f(y_t)-\gamma_t(\nu_{t-1}^{-1}-1)f(y_{t-1})-\gamma_tf_*.
\end{align*}
Besides this, for $t=1$, we have $\gamma_1\left< f'(y_1) \middle| x_1-x_*
\right> \geqslant \gamma_1(f(y_1)-f_*)$ because $x_1=y_1$ by definition.
Then, summing over $t=1,\dots,T$, we obtain after simplifications:
\begin{multline*}
(\gamma_1-\gamma_2(\nu_1^{-1}-1))f(y_1)+\sum_{t=2}^{T-1}(\gamma_t\nu_{t-1}^{-1}-\gamma_{t+1}(\nu_t^{-1}-1))f(y_t)+\gamma_T\nu_{t-1}^{-1}f(y_T)\\
-\left( \sum_{t=1}^T\gamma_t \right)f_*\leqslant \sum_{t=1}^T\left< \gamma_tf'(y_t) \middle| x_t-x_* \right> .
\end{multline*}
Using the definition of coefficients $\nu_t$, the above left-hand side
simplifies to result in the inequality
\[ \left( \sum_{t=1}^T\gamma_t \right)\left( f(y_T)-f_* \right)\leqslant \sum_{t=1}^T\left< \gamma_tf'(y_t) \middle| x_t-x_* \right>. \]
Finally, because $(x_t,\vt_t)_{t\geqslant 1}$ is a sequence of UMD$(h,\xi)$ iterates
with dual increments $\xi:=(-\gamma_tf'(y_t))_{t\geqslant 1}$,
the result then follows by applying inequality \eqref{eq:regret-bound-norm} from Corollary~\ref{cor:regret-bound} and dividing by $\sum_{t=1}^T\gamma_t$.\qed

\subsubsection{Proof of Theorem \ref{thm:smooth-convex}}
First, observe that whenever $\gamma_t\leq 1/L$, due to \eqref{eq:lip-grad},
\be f(x_{t+1})-f(x_{t})-\ps{\nabla f(x_t)}{x_{t+1}-x_t}&\leq& {L\over 2}\|x_{t+1}-x_t\|^2\nonumber\\&\leq& (2\gamma_t)^{-1}\|x_{t+1}-x_t\|^2.
\ee{eq:lip2}
Thus,
\begin{align*}
\gamma_t D_f(x_{t+1},x_t)&=\gamma_t[f(x_{t+1})-f(x_{t})-\ps{f'(x_t)}{x_{t+1}-x_t}]\\
&\leq \half\|x_{t+1}-x_t\|^2\\&\leq D_h(x_{t+1},x_t;\ \vt_t).
\end{align*}
by the strong convexity of $D_h$. On the other hand,  by \eqref{eq:3-point} of Lemma \ref{lm:3-point}, for any $x\in \X\cap \operatorname{dom}h$,
\begin{align*}
D_h(x,x_{t+1}; \vt_{t+1})&\leqslant D_h(x,x_t;\vt_t)+\gamma_t\ps{\nabla f(x_t)}{x-x_{t+1}}-D_h(x_{t+1},x_t; \vt_t)\\
\mbox{[by
  \eqref{eq:back}]}&\leq D_h(x,x_t;\vt_t)+\gamma_t\ps{\nabla f(x_t)}{x-x_{t}}\\
                         &\qquad -\gamma_t\ps{\nabla f(x_t)}{x_{t+1}-x_{t}}-D_f(x_{t+1},x_t)\\
\mbox{[due to \eqref{eq:lip2}]}&\leq D_h(x,x_t;\vt_t)+\gamma_t\ps{\nabla f(x_t)}{x-x_{t}}-\gamma_t[f(x_{t+1}-f(x_t)]\\
\mbox{[by convexity of $f$]}
&\leq D_h(x,x_t;\vt_t)-\gamma_t(f(x_{t+1})-f(x)).
\end{align*}
Consequently, $\forall x\in \X\cap \operatorname{dom}h$,
\[
\gamma_t(f(x_{t+1})-f(x_t))\leq D_h(x,x_t;\vt_t)-D_h(x,x_{t+1}; \vt_{t+1}).
\]
When applying the above inequality to $x=x_t$ we conclude that \[\gamma_t(f(x_{t+1})-f(x_t))\leq -D_h(x_t,x_{t+1}; \vt_{t+1})\leq 0.\]
Finally, when setting $x=x_*$, we obtain
\[
\left(\sum_{t=1}^T\gamma_t\right) (f(x_{T+1})-f_*)\leq \sum_{t=1}^T\gamma_t(f(x_{t+1})-f(x_*))\leq D_h(x_*,x_1;\vt_1)
\]
which implies \eqref{eq:smooth_0}. \qed

\subsubsection{Proof of Theorem \ref{thm:acceleration}}

We start with the following technical result.
\begin{lemma}\label{lemma:lan}
Assume that positive step-sizes $\nu_t\in(0,1]$ and $\gamma_t>0$ are such that
the relationship
\be
f({z}_{t+1})\leq f(y_t)+\nu_t\ps{\nabla f(y_t)}{{x}_{t+1}-x_t}+{\nu_t\over \gamma_t} D_h({x}_{t+1},{x}_t;\vt_t),
\ee{eq: keep_nest1}
holds for all $t$ which is certainly the case if $\nu_t\gamma_t\leq L^{-1}$. Denote $s_t=f(z_t)-f_*$; then
\begin{align}
{\gamma_t\nu_t^{-1}}(s_{t+1}-s_t)+\gamma_ts_t&\leq D_h(x_*,{x}_t;\vt_t)-D_h(x_*,{x}_{t+1};\vt_{t+1}).
\label{eq: lai1}
\end{align}
\end{lemma}
\noindent {\bf Proof of the lemma.}
Observe first that by construction,
\[
{z}_{t+1}-y_t= (1-\nu_t){z}_t+\nu_t{x}_{t+1}-[(1-\nu_t){z}_t+\nu_t{x}_{t}]=\nu_t({x}_{t+1}-{x}_t)
\]
By strong convexity of $h$, for $\nu_t\gamma_t\leq L^{-1}$ we have
\begin{align*}
f({z}_{t+1})&\leq f(y_t)+\langle \nabla f(y_t),{z}_{t+1}-y_t\rangle+{L\over 2}\|{z}_{t+1}-y_t\|^2\nn
&= f(y_t)+\nu_t\langle \nabla f(y_t),{x}_{t+1}-x_t\rangle+{L\nu_t^2\over 2}\|{x}_{t+1}-{x}_t\|^2\nn
&\leq f(y_t)+\nu_t\langle \nabla f(y_t),{x}_{t+1}-x_t\rangle+{\nu_t\over \gamma_t}D_h({x}_{t+1},{x}_t;\vt_t),
\end{align*}
what is \eqref{eq: keep_nest1}.

Next, observe that by \eqref{eq:acc-y},
\[
\nu_t(x_*-x_t)=(\nu_tx_*+(1-\nu_t)z_t)-y_t,
\]
whence, by convexity of $f$,
\begin{align*}
\nu_t\ps{\nabla f(y_t)}{x_*-x_t}&=\ps{\nabla f(y_t)}{(\nu_tx_*+(1-\nu_t)z_t)-y_t}\\
                                &\leq f(\nu_tx_*+(1-\nu_t)z_t)-f(y_t)\\
  &\leq \nu_t(f(x_*)-f(y_t))+(1-\nu_t)(f(z_t)-f(y_t)).
\end{align*}
When substituting the latter bound into \eqref{eq: keep_nest1} we get
\begin{align*}
f({z}_{t+1})&\leq f(y_t)+\nu_t\ps{\nabla f(y_t)}{{x}_{t+1}-x_*}+\nu_t(f(x_*)-f(y_t))\\&\qquad +(1-\nu_t)(f(z_t)-f(y_t))
+{\nu_t\over \gamma_t} D_h({x}_{t+1},{x}_t;\vt_t),
\end{align*}
or
\begin{align*}
f({z}_{t+1})-f(z_t)\leq \nu_t\ps{\nabla f(y_t)}{{x}_{t+1}-x_*}+\nu_t(f_*-f(z_t))
+{\nu_t\over \gamma_t} D_h({x}_{t+1},{x}_t;\vt_t).
\end{align*}
Now, because $(x_t,\vt_t)_{t\geqslant 1}$ is a sequence of UMD iterates,  by  \eqref{eq:3-point} of Lemma \ref{lm:3-point},
\[
\gamma_t\ps{\nabla f(y_t)}{{x}_{t+1}-x_*}\leq D_h(x_*,{x}_t;\vt_t)-D_h(x_*,{x}_{t+1};\vt_{t+1})-D_h({x}_{t+1},{x}_t;\vt_t),
\]
and we arrive at
\[
{\gamma_t\nu_t^{-1}}(f({z}_{t+1})-f(z_t))\leq D_h(x_*,{x}_t;\vt_t)-D_h(x_*,{x}_{t+1};\vt_{t+1})+\gamma_t(f_*-f(z_t)),
\]
what is \eqref{eq: lai1}.\qed 
\paragraph{Proof of the theorem.} Assume that $\nu_t$ and $\gamma_t$ satisfy
\be
\nu_1=1,\quad \nu_t\in (0,1], \quad\gamma_{t+1}(\nu_{t+1}^{-1}-1)\leq \gamma_t\nu_{t}^{-1}.
\ee{eq: keep_nest0}
When summing \eqref{eq: lai1} up from 1 to $T$ we get
\begin{align*}
D_h(x_*,x_t;\vt_1)&\geq \sum_{t=1}^T[{\gamma_t\nu_t^{-1}}(s_{t+1}-s_t)+\gamma_ts_t]\\
&={\gamma_T\nu_T^{-1}}s_{T+1}+\sum_{t=2}^Ts_t\left({\gamma_{t-1}\nu_{t-1}^{-1}}-\gamma_t(\nu_t^{-1}-1)\right)-\gamma_1(\nu_1^{-1}-1)s_1\\
&\!\!\!\!\!\!\!\!\underbrace{\geq}_{\mbox{[by \eqref{eq: keep_nest0}]}} {\gamma_T\nu_T^{-1}}s_{T+1}={\gamma_T\nu_T^{-1}}(f(z_{T+1})-f_*).
\end{align*}
It is clear that the choice of $\gamma_1=L^{-1}$, $\nu_1=1$ and $\nu_t=(\gamma_tL)^{-1}$ satisfies the relationship $\gamma_t\nu_t\leq L^{-1}$.
In this case, when choosing step-sizes $(\gamma_t)_{t\geq 1}$ to saturate recursively the last relation in \eqref{eq: keep_nest0}, specifically,
\[
\gamma^2_{t+1}L-\gamma_{t+1}= \gamma_t^2L
\]
we come to celebrated Nesterov step-sizes \eqref{eq:nest-steps} which satisfy $\gamma_t\nu_t^{-1}\geq {(t+1)^2\over 4L}$, and we arrive at \eqref{eq:nest-bound}.\qed

\end{document}